\documentclass[a4paper,12pt]{article}
\usepackage{graphicx}
\usepackage{latexsym}
\usepackage{amsmath}
\usepackage{amssymb}
\usepackage{enumerate}
\usepackage{pict2e}
\usepackage{xcolor}
\usepackage{tikz}
\usepackage{float}
\usepackage{cancel}
\usepackage{theorem}
\newtheorem{theorem}{Theorem}[section]
\newtheorem{Proposition}[theorem]{Proposition}
\newtheorem{Lemma}[theorem]{Lemma}
\newtheorem{Corollary}[theorem]{Corollary}

\theorembodyfont{\rmfamily}
\newtheorem{proof}{\textmd{\textit{Proof.}}}

\newtheorem{Remark}[theorem]{Remark}
\newtheorem{Example}[theorem]{Example}
\newtheorem{Definition}[theorem]{Definition}

\makeatletter

\@addtoreset{equation}{section}
\makeatother

\newcommand{\R}{\ensuremath{\mathbb{R}}}
\newcommand{\Sph}{\ensuremath{\mathbb{S}}}

\newcommand{\cL}{\ensuremath{\mathcal{L}}}

\newcommand{\cP}{\ensuremath{\mathcal{P}}}

\title{The geometry of a Randers rotational surface with an arbitrary direction wind}
\author{Rattanasak Hama and Sorin V. Sabau}

\date{\today}
\begin{document}
\footnotetext{MATHEMATICS, Volume 8 Issue 11 Published 2020}
\maketitle

\section{Introduction}
A Finsler structure on a surface $M$ can be regarded as a smooth 3-manifold $\Sigma\subset TM$ for which the canonical projection $\pi:\Sigma\to M$ is a surjective submersion and having the property that for each $x\in M$, the $\pi$-fiber $\Sigma_x=\pi^{-1}(x)$ is a strictly convex curve including the origin $O_x\in T_xM$. Here we denote by $TM$ the tangent bundle of $M$. This is actually equivalent to saying that such a geometrical structure $(M,F)$ is a surface $M$ endowed with a Minkowski norm in each tangent space $T_xM$ that varies smoothly with the base point $x\in M$ all over the manifold. Obviously $\Sigma$ is the unit sphere bundle $\{(x,y)\in TM:F(x,y)=1\}$, also called the indicatrix bundle. Even though the these notions are defined for arbitrary dimension, we restrict to surfaces hereafter (\cite{BCS}).

On the other hand, such a Finsler structure defines a 2-parameter family of oriented paths on $M$, one in every oriented direction through every point. This is a special case of the notion of path geometry. We recall that, roughly speaking, a path geometry on a surface $M$ is a 2-parameter family of curves on $M$ with the property that through each point $x \in M$ and in each tangent direction at $x$ there passes a unique curve in the family. The fundamental example to keep in mind is the family of lines in the Euclidean plane.

To be more precise, a path geometry on a surface $M$ is a foliation $\mathcal P$ of the projective tangent bundle $\mathbb{P} TM$ by contact curves, each of which is transverse to the fibers of the canonical projection $\pi:\mathbb{P} TM\to M$. Observe that even though $\mathbb{P} TM$ is independent of any norm $F$, actually there is a Riemannian isometry between $\mathbb{P} TM$ and $\Sigma$, fact that allows us to identify them in the Finslerian case(\cite{B}).

The 3-manifold  $\mathbb{P} TM$ is naturally endowed with a contact structure. Indeed, observe that for a curve
be a smooth, immersed curve  $\gamma:(a,b)\to M$, let us denote by $\hat{\gamma}:(a,b)\to \mathbb{P}TM$ its canonical lift to the projective tangent bundle $\mathbb{P} TM$. Then, the fact that the canonical projection  is a submersion implies that, for each line $L \in \mathbb P TM $, the linear map $\pi_{*,L} : T _L \mathbb P TM  \to T_x M$, is surjective, where $\pi  ( L ) = x \in M$. Therefore $E_L := \pi_{*,L}^{ -1} ( L ) \subset T_L \mathbb P TM$ is a 2-plane in $T_L \mathbb P TM$ that defines a contact distribution and therefore a contact structure on $\mathbb P TM$. A curve on  $\mathbb P TM$ is called contact curve if it is tangent to the contact distribution $E$. Nevertheless, the canonical lift $\hat{\gamma}$ to $\mathbb P TM$ of a curve $\gamma$ on $M$ is a contact curve.

A local path geometry on $M$ is a foliation $\mathcal P$ of an open subset $U \subset \mathbb P TM$ by contact curves, each of which is transverse to the fibers of $\pi : \mathbb P TM  \to M$.

If $(M,F)$ is a Finsler surface, then the 3-manifold $\Sigma$ is endowed with a canonical coframe $(\omega^1,\omega^2,\omega^3)$ satisfying the structure equations
\begin{equation}\label{struct eq F}
	\begin{split}
		d\omega^1 & = -I\omega^1\wedge \omega^3+\omega^2\wedge \omega^3  \\
		d\omega^2 & =  \omega^3\wedge\omega^1 \\
		d\omega^3 & =   K\omega^1\wedge\omega^2-J\omega^1\wedge\omega^3,\\
	\end{split}
\end{equation}
where the functions $I,J$ and $K:TM\to \R$ are the Cartan scalar, the Landsberg curvature and the Gauss curvature, respectively. 
The 2-plane field $D:=\langle \hat{e}_2,\hat{e}_3\rangle$ defines a contact structure on $\Sigma$, where we denote $(\hat{e}_1,\hat{e}_2,\hat{e}_3)$ the dual frame of $(\omega^1,\omega^2,\omega^3)$. Indeed, it can be seen that the 1-form $\eta:=A \omega^1$ is a contact form for any function $A\neq 0$ on $\Sigma$. The structure equations \eqref{struct eq F} imply $\eta\wedge d\eta=A^2\omega^1\wedge \omega^2
\wedge \omega^3\neq 0$. Observe that in the Finslerian case, we actually have two foliations on the 3-manifold $\Sigma$:
\begin{enumerate}
	\item $\mathcal P=\{\omega^1=0,\omega^3=0\}$ the geodesic foliation of $\Sigma$, i.e. the leaves are curves in $\Sigma$ tangent to the geodesic spray $\hat{e}_2$;
	\item $\mathcal Q=\{\omega^1=0,\omega^2=0\}$ the indicatrix foliation of $\Sigma$, i.e. the leaves are indicatrix curves in $\Sigma$ tangent $\hat{e}_3$.
\end{enumerate}
The pair $(\mathcal P,\mathcal Q)$ is called sometimes a generalized path geometry (see \cite{Br}).

The {\it (forward) integral length} of a regular piecewise $C^{\infty}$-curve $\gamma:[a,b]\to M$ on a Finsler surface  $(M,F)$ is given by
\begin{equation*}\label{integral length}
	{\cal L}_{\gamma}:=\sum_{i=1}^k\int_{t_{i-1}}^{t_i}F(\gamma(t),\dot\gamma(t))dt,
\end{equation*}
where $\dot\gamma=\frac{d\gamma}{dt}$ is the tangent vector along the curve $\gamma|_{[t_{i-1},t_i]}$.

A regular  piecewise $C^\infty$-curve $\gamma$ on a Finsler manifold is called a {\it forward geodesic} if $({\cL_\gamma})'(0)=0$ for all piecewise $C^\infty$-variations of $\gamma$ that keep its ends fixed. In terms of Chern connection a constant speed geodesic is characterized by the condition $D_{\dot\gamma}{\dot\gamma}=0$. Observe that the canonical lift of a geodesic $\gamma$ to $\mathbb PTM$ gives the geodesics foliation $\mathcal P$ described above. 

Using the integral length of a curve, one can define the Finslerian distance between two points on $M$. For any two points $p$, $q$ on $M$, let us denote by $\Omega_{p,q}$ the set of all piecewise $C^\infty$-curves $\gamma:[a,b]\to M$ such that $\gamma(a)=p$ and $\gamma(b)=q$. Then the map
\begin{equation*}
	d:M\times M\to [0,\infty),\qquad d(p,q):=\inf_{\gamma\in\Omega_{p,q}}{\cal L}_{\gamma}
\end{equation*}
gives the {\it Finslerian distance} on $M$. It can be easily seen that $d$ is in general a quasi-distance, i.e., it has the properties $d(p,q)\geq 0$, with equality if and only if $p=q$, and $d(p,q)\leq d(p,r)+d(r,q)$, with equality if and only if  $r$ lies on a minimal geodesic segment joining from $p$ to $q$ (triangle inequality).

A Finsler manifold $(M,F)$ is called {\it forward geodesically complete} if and only if any short geodesic $\gamma:[a,b)\to M$ can be extended to a long geodesic $\gamma:[a,\infty)\to M$. The equivalence between forward completeness as metric space and geodesically completeness is given by the Finslerian version of Hopf-Rinow Theorem (see for eg. \cite{BCS}, p. 168). Same is true for backward geodesics. 
In the Finsler case, unlikely the Riemannian counterpart, forward completeness is not equivalent to backward one, except the case when $M$ is compact.

Any geodesic $\gamma$ emanating from a point $p$ in a compact Finsler manifold loses the global minimising
property at a point $q$ on $\gamma$. Such a point $q$ is called a cut point of $p$ along $\gamma$. The cut locus of a
point $p$ is the set of all cut points along geodesics emanating from $p$. This kind of points often appears
as an obstacle when we try to prove some global theorems in differential geometry being in the same time  vital  in analysis, where appear as a singular points set. In fact, the cut locus of a point $p$ in a complete Finsler manifold equals the closure of the set of all
non-differentiable points of the distance function from the point $p$. The structure of the cut locus
plays an important role in optimal control problems in space and quantum dynamics allowing to
obtain global optimal results in orbital transfer and for Lindblad equations in quantum control.

The notion of cut locus was introduced and studied for the first time by H. Poincare in 1905 for the Riemannian case. In the case of a two dimensional analytical sphere, S. B. Myers has proved in 1935 that the cut locus of a point is a finite tree  in both Riemannian and Finslerian cases. In the case of an analytic Riemannian manifold, M. Buchner has shown the triangulability of the cut locus of a point $p$, and has determined its local structure for the low dimensional case in 1977 and 1978, respectively. The cut locus of a point can have a very complicated structure. For example, H. Gluck and D. Singer have constructed a $C^\infty$ Riemannian manifold that has a point whose cut locus is not triangulable (see \cite{SST} for an exposition). There are $C^k$-Riemannian or Finsler metrics on spheres with a preferential point  whose cut locus is a fractal (\cite{IS}).

In the present paper we will study the local and global behaviour of the geodesics of a Finsler metric of revolution on topological cylinders. In special, we will determine the structure of the cut locus on the cylinder for such metrics and compare it with the Riemannian case.

Will focus on the Finsler metrics of Randers type obtained as solutions of the Zermelo's navigation problem for the navigation data $(M,h)$, where $h$ is the canonical Riemannian metric on the topological cylinder $h=dr^2+m^2(r)d\theta^2$,  and $W=A(r)\frac{\partial}{\partial r}+B\frac{\partial}{\partial \theta}$ is a vector field on $M$. Observe that our wind is more general than a Killing vector field, hence our theory presented here is a generalization of the classical study of geodesics and cut locus for Randers metrics obtained as solutions of the Zermelo's navigation problem with Killing vector fields studied in \cite{HCS} and \cite{HKS}. Nevertheless, by taking the wind $W$ in this way we obtain a quite general Randers metric on $M$ which is a Finsler metric of revolution and whose geodesics and cut locus can be computed explicitly.  

Our paper is organized as follows. In the Section \ref{sec_Randers_theory}, we recall basics of Finsler geometryusing  the Randers metrics that we will actually use in order to obtain explicit information on the geodesics behaviour and cut locus structure. We introduce an extension of the Zermelo's navigation problem for Killing winds to a more general case $\widetilde{W}=V+W$, where only $W$ is Killing. We show that the geodesics, conjugate locus and cut locus can be determined in this case as well.

In the section \ref{sec:Surf of revol} we describe the theory of general Finsler surfaces of revolution. In the case this Finsler metric is a Riemannian one, we obtain the theory of geodesics and cut locus known already (\cite{C1}, \cite{C2}). 

In the Section \ref{sec_Randers_metric} we consider some examples that ilustrate the theory depicted until here. In particular,        in subsection \ref{sec_A(r),B} we consider the general wind $W=A(r)\frac{\partial}{\partial r}+B\frac{\partial}{\partial \theta}$ which obviously is not Killing with respect to $h$, where $A=A(r)$ is a bounded function and $B$ is a constant and determine its geometry here. Essentially, we are reducing the geodesics theory of the Finsler metric $\widetilde{F}$, obtained from the Zermelo's navigation problem for $(M,h)$ and $\widetilde{W}$, to the theory of a Riemannian metric $(M,\alpha)$. 

Moreover, in the particular case  $\widetilde{W}=A\frac{\partial}{\partial r}+B\frac{\partial}{\partial \theta}$ in Section \ref{sec_A,B}, where $A,B$ are  constants, the geodesic theory of  $\widetilde{F}$ can be directly obtained from the geometry of the Riemannian metric $(M,h)$. A similar study can be done for the case  $W=A(r)\frac{\partial}{\partial r}$. We leave a detailed study of these Randers metrics to a forthcoming research. 


\section{Finsler metrics. The Randers case}\label{sec_Randers_theory}

Finsler structures are one of the most natural generalization of Riemannian metrics. Let us recall here that a Finsler structure on a real smooth $n$-dimensional manifold $M$ is a function $F:TM\to [0,\infty)$ which is smooth on $\widetilde{TM}=TM\setminus O$, where $O$ is the zero section, has the {\it homogeneity property} $F(x,\lambda y)=\lambda F(x,y)$, for all $\lambda>0$ and all $y\in T_xM$ and also has the strong convexity property that the Hessian matrix
\begin{equation}
	g_{ij}=\frac{1}{2}\frac{\partial^2F^2}{\partial y^i\partial y^j}(x,y)
\end{equation}
is positive definite at any point $(x,y)\in TM$. 

\subsection{An ubiquitous family of Finsler structures: the Randers metrics}\label{sec_ubiquitous}

Initially introduced in the context of general relativity, Randers metrics are the most ubiquitous family of Finsler structures.

A {\it Randers metric} on a surface $M$ is obtained by a rigid translation  of an ellipse in each tangent plane $T_xM$ such that the origin of $T_xM$ remains inside it. 

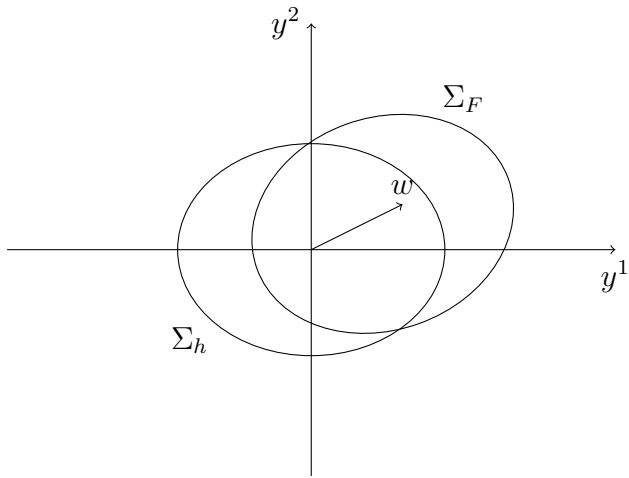
\begin{figure}[h]
	\begin{center}
		\setlength{\unitlength}{1cm}
		\begin{tikzpicture}
			\draw[rotate=20,xshift=1cm] (0,0) ellipse (50pt and 40pt);
			\draw[rotate=0,xshift=0cm] (0,0) ellipse (50pt and 40pt);
			\draw [->] (0,0) -- (1.2,0.6) node[above]{$w$};
			\draw [->] (-4,0) -- (4,0) node[below]{$y^1$};
			\draw [->] (0,-3) -- (0,3) node[left]{$y^2$};
			\node at (-1.6,-1.2) {$\Sigma_h$};
			\node at (2,2) {$\Sigma_F$};
		\end{tikzpicture}	
	\end{center}
	\caption{Randers metrics: a rigid dispacement of an ellipse.}
	\label{fig_1}
\end{figure}

Formally, on a Riemannian manifold $(M,\alpha)$, a Randers metric is a Finsler structure $(M,F)$ whose fundamental function $F:TM\to [0,\infty)$ can be written as

\begin{equation*}
	F(x,y)=\alpha(x,y)+\beta(x,y),
\end{equation*}
where $\alpha(x,y)=\sqrt{a_{ij}(x)y^iy^j}$ and $\beta(x,y)=b_i(x)y^i$, such that the Riemannian norm of $\beta$ is less than 1, i.e. $b^2:=a^{ij}b_ib_j<1$. 

It is known that Randers metrics are solutions of the {\it Zermelo's navigation problem} \cite{Z} which we recall here.

{\it Consider a ship sailing on the open sea in calm waters. If a mild breeze comes up, how should the ship be steered in order to reach a given destination in the shortest time possible? 
}

The solution  was given by Zermelo in the case the open sea is an Euclidean space, by \cite{Sh} in the Riemannian case and studied in detailed in \cite{BRS}.

Indeed, for a time-independent wind $W\in TM$, on a Riemannian manifold $(M,h)$,  the paths minimizing travel-time are exactly the geodesics of the Randers metric 
\begin{equation*}
	F(x,y)=\alpha(x,y)+\beta(x,y)=\frac{\sqrt{\lambda\|y\|_h^2+W_0}}{\lambda}-\frac{W_0}{\lambda},
\end{equation*}
where $W=W^i(x)\frac{\partial}{\partial x^i}$, $\|y\|_h^2=h(y,y)$, $\lambda=1-|W|_h^2$, and $W_0=h(W,y)$. Requiring $\|W\|_h<1$ we obtain a positive definite Finslerian norm. In components, $a_{ij}=\frac{1}{\lambda}h_{ij}+\frac{W_i}{\lambda}$, $b_i(x)=-\frac{W_i}{\lambda}$, where $W_i=h_{ij}W^j$ (see \cite{R} for a general discussion). 

The Randers metric obtained above is called {\it the solution of the Zermelo's navigation problem for the navigation data $(M,h)$ and $W$}. 

\begin{Remark}
	Obviously,  at any $x\in M$, the condition $F(y)=1$ is equivalent to $\|y-W\|_h=1$ fact that assures that, indeed, the indicatrix of $(M,F)$ in $T_xM$ differs from the unit sphere of $h$ by a translation along $W(x)$ (see Figure \ref{fig_1}). 
\end{Remark}

More generally, the Zermelo's navigation problem can be considered where the open sea is a given Finsler manifold (see \cite{Sh}). 

We have
\begin{Proposition}
	Let $(M,F)$ be a Finsler manifold and $W$ a vector field on $M$ such that $F(-W)<1$. Then the solution of the Zermelo's navigation problem with navigation data $F,W$ is th Finsler metric $\widetilde{F}$ obtained by solving the equation
	\begin{equation}\label{eq_*1}
		F(y-\widetilde{F}W)=\widetilde{F}, \ \text{for any} \ y\in TM.
	\end{equation}
\end{Proposition}

Indeed, if we consider the Zermelo's navigation problem where the open sea is the Finsler manifold $(M,F)$ and the wind $W$, by rigid translation of the indicatrix $\Sigma_F$ we obtain the closed, smooth, strongly convex indicatrix $\Sigma_{\widetilde{F}}$, where $\widetilde{F}$ is solution of the equation $F\left(\frac{y}{\widetilde{F}}-W\right)=1$ which is clearly equivalent to \eqref{eq_*1} due to positively of $\widetilde{F}$ and homogeneity of $F$.

To get a genuine Finsler metric $\widetilde{F}$, We need for the origin $O_x\in T_xM$ to belong to the interior of $\Sigma_{\widetilde{F}}=\Sigma_F+W$, that is $F(-W)<1$.

\begin{Remark}\label{rem_F_positive}
	Consider the Zermelo's navigation problem for $(M,F)$ and wind $W$, where $F$ is a (positive-defined) Finsler metric. If we solve the equation
	$$
	F\left(\frac{y}{\widetilde{F}}-W\right)=1\Leftrightarrow F(y-\widetilde{F}W)=\widetilde{F}
	$$
	let $\widetilde{F}$ we obtain the solution of this Zermelo's navigation problem.
	
	In order that $\widetilde{F}$ is Finsler we need to check:
	\begin{itemize}
		\item[(i)] $\widetilde{F}$ is strongly convex
		\item[(ii)] the indicatrix of $\widetilde{F}$ includes the origin
		$O_x\in T_xM$.
	\end{itemize}
	
	Since indicatrix of $\widetilde{F}$ is the rigid translation by $W$ of the indicatrix of $F$, and indicatrix of $F$ is strongly convex, it follows indicatrix of $\widetilde{F}$ is also strongly convex.
	
	Hence, we need to find the condition for (ii) only.
	
	Denote
	$$
	B_F(1):=\{y\in T_xM:F(y)<1\},\quad 
	\widetilde{B}_{\widetilde{F}}(1):=\{y\in T_xM:\widetilde{F}(y)<1\}
	$$
	the unit balls of $F$ and $\widetilde{F}$, respectively.
	
	The Zermelo's navigation problem shows
	$$
	B_{\widetilde{F}}(1)=B_F(1)+W.
	$$
	
	Hence
	$$
	O_x\in B_{\widetilde{F}}(1)\Leftrightarrow O_x\in B_F(1)+W\Leftrightarrow -W\in B_F(1)\Leftrightarrow F(-W)<1.
	$$
	
	Hence, indicatrix of $\widetilde{F}$ include $O_x \Leftrightarrow F(-W)<1$, where we denote by $O_x$ the zero vector.
\end{Remark}

\begin{Proposition}\label{prop_2steps_Zermelo}
	Let $(M,F_1=\alpha+\beta)$ be a Randers space and $W=W^i(x)\frac{\partial}{\partial x^i}$ a vector field on $M$. Then, the solution of the Zermelo's navigation problem with navigation data $(M,F_1)$ and $W$ is also a Randers metric $F=\widetilde{\alpha}+\widetilde{\beta}$, where
	\begin{equation}\label{eq_tilde a tilde b}
		\begin{split}
			\widetilde{a}_{ij}&= \frac{1}{\eta}\left(a_{ij}-b_ib_j\right)+\left(\frac{W_i-b_i[1+\beta(W)]}{\eta}\right)\left( \frac{W_j-b_j[1+\beta(W)]}{\eta}\right)\\
			\widetilde{b}_{i}&= -\frac{W_i-b_i[1+\beta(W)]}{\eta}, 
		\end{split}
	\end{equation}
	where $\eta=[1+\beta(W)]^2-\alpha^2(W)$, $W_i=a_{ij}W^j$.
\end{Proposition}
\begin{proof}[Proof of Proposition \ref{prop_2steps_Zermelo}]
	Let us consider the equation
	$$
	F_1\left(\frac{y}{\widetilde{F}}-W\right)=1
	$$
	which is equivalent to
	$$
	F_1(y-\widetilde{F}W)=\widetilde{F}
	$$
	due to positively of $\widetilde{F}$ and 1-positive homogeneity of $F_1$.
	
	If we use $F_1=\alpha+\beta$, it follows
	$$
	\alpha(y-\widetilde{F}W)=\widetilde{F}-\beta(y-\widetilde{F}W),
	$$
	using the linearity of $\beta$, i.e. $\beta(y-\widetilde{F}W)=\beta(y)-\widetilde{F}\beta(W)$, where $\beta(y)=b_iy^i$, $\beta(W)=b_iW^i$, and squaring this formula, we get the equation
	\begin{equation}\label{eq_3*}
		\alpha^2(y-\widetilde{F}W)=[\widetilde{F}(1+\beta(W))-\beta(y)]^2.
	\end{equation}
	
	Observe that
	\begin{equation}\label{eq_1*}
		\alpha^2(y-\widetilde{F}W)=\alpha^2(y)-2\widetilde{F}<y,W>_\alpha+\widetilde{F}^2\alpha^2(W)
	\end{equation}
	and
	\begin{equation}\label{eq_2*}
		[\widetilde{F}-\beta(y-\widetilde{F}W)]^2=[1+\beta(W)]^2\widetilde{F}^2-2\widetilde{F}\beta(y)[1+\beta(W)]+\beta^2(y),	
	\end{equation}
	substituting \eqref{eq_1*}, \eqref{eq_2*} in \eqref{eq_3*} gives the 2nd degree equation
	\begin{equation}\label{eq_4*}
		\eta \widetilde{F}^2+2\widetilde{F}<y,\quad W-B[1+\beta (W)]>_\alpha-[\alpha^2(y)-\beta^2(y)]=0,
	\end{equation}
	where $B=b^i\frac{\partial}{\partial x^i}=(a^{ij}b_j)\frac{\partial}{\partial x^i}$ and $\eta:=[1+\beta(W)]^2-\alpha^2(W)$, i.e.
	$$
	<y,W-B[1+\beta(W)]>_\alpha=a_{ij}y^i(w^j-b^j[1+\beta(W)])=<y,W>_\alpha-\beta(y)[1+\beta(W)].
	$$
	
	The discriminant of \eqref{eq_4*} is
	$$
	D'=\{<y,W>_\alpha-\beta(y)[1+\beta(W)]\}^2+\eta[\alpha^2(y)-\beta^2(y)].
	$$
	
	Let us observe that $F_1(-W)<1$ implies $\eta>0$. Indeed
	$$
	F_1(-W)=\alpha(W)-\beta(W)<1 \Leftrightarrow \alpha^2(W)<[1+\beta(W)]^2
	$$
	hence $\eta>0$.
	
	Moreover, observe that
	$$
	D'=\{\eta(a_{ij}-b_ib_j)+(w_i-b_i[1+\beta(W)])(w_j-b_j[1+\beta(W)])\}y^iy^j.
	$$
	
	The solution of \eqref{eq_4*} is given by
	$$
	\widetilde{F}=\frac{\sqrt{<y,W-B[1-\beta(W)]>_\alpha^2+\eta[\alpha^2(y)-\beta^2(y)]}}{\eta}-\frac{<y,W-B[1-\beta(W)]>_{\alpha}}{\eta}
	$$
	or equivalently to
	$$
	\widetilde{F}=\frac{\sqrt{\{\eta(a_{ij}-b_ib_j)+(w_i-b_i[1+\beta(W)])(w_j-b_j[1+\beta(W)])\}y^iy^j}}{\eta}-\frac{\{W_i-b_i[1+\beta(W)]\}y^i}{\eta},
	$$
	that is $\widetilde{F}=\widetilde{\alpha}+\widetilde{\beta}$, where $\widetilde{a}_{ij}$ and $\widetilde{b}_i$ are given by \eqref{eq_tilde a tilde b}.
	
	Observe that $\widetilde{a}_{ij}$ is positive defined. Indeed, for any $v\in TM$, $\widetilde{\alpha}^2(v,v)=\widetilde{a}_{ij}v^iv^j=\eta[\alpha^2(v)-\beta^2(v)]+<v,W-B[1+\beta(W)]>^2$.
	
	On the other hand, since $F_1=\alpha+\beta$ is Randers metric, $F_1(X)>0$ for any tangent vector $X\in TM$, hence for $X=v$ and $X=-v$ we get $\alpha(v)+\beta(v)>0$ and $\alpha(v)-\beta(v)>0$, respectively, hence $\alpha^2(v)-\beta^2(v)>0$ for any $v\in TM$.
	
	This implies $\widetilde{a}_{ij}$ is positive defined.
\end{proof}

\subsection{A two steps Zermelo's navigation}\label{sec_two_steps_Zermelo}

We have discussed in the previous section the Zermelo's navigation when the open sea is a Riemannian manifold and when it is a Finsler manifold, respectively.

In order to obtain a more general version of the navigation, we combine these two approaches. We have

\begin{theorem}\label{thm_two_steps_Zermelo}
	Let $(M,h)$ be a Riemannian manifold and $V$, $W$ two vector fields on $M$.
	
	Let us consider the Zermelo's navigation problem on $M$ with the following data
	\begin{enumerate}
		\item[(I)] Riemannian metric $(M,h)$ with wind $V+W$ and assume condition $\|V+W\|_h<1$;
		\item[(II)] Finsler metric $(M,F_1)$ with wind $W$ and assume $W$ satisfies condition $F_1(-W)<1$, where $F_1=\alpha+\beta$ is the solution of the Zermelo' s navigation problem for the navigation data $(M,h)$ with wind $V$, such that $\|V\|_h<1$.
	\end{enumerate}
	Then, the above Zermelo's navigation problems (I) and (II) have the same solution $F=\widetilde{\alpha}+\widetilde{\beta}$. 
\end{theorem}

\begin{proof}[Proof of Theorem \ref{thm_two_steps_Zermelo}]
	
	Let us consider case (I), i.e. the sea is the Riemannian metric $(M,h)$ with the wind $\widetilde{W}:=V+W$ such that $\|V+W\|_h<1$. The associated Randers metric through the Zermelo's navigation problem is given by $\widetilde{\alpha}+\widetilde{\beta}$, where
	\begin{equation}\label{eq_1.1*}
		\begin{split}
			\widetilde{a}_{ij}:=\frac{1}{\Lambda}h_{ij}+\left(\frac{\widetilde{W}_i}{\Lambda}\right)\left(\frac{\widetilde{W}_j}{\Lambda}\right), \ \widetilde{b}_i:=-\frac{\widetilde{W}_i}{\Lambda},
		\end{split}
	\end{equation}
	where $\Lambda=1-\|\widetilde{W}\|^2_h=1-\|V+W\|^2_h$, $\widetilde{W}_i=h_{ij}\widetilde{W}^j$.
	
	Observe that \eqref{eq_1.1*} are actually equivalent to
	\begin{equation}\label{eq_1.2*}
		\begin{split}
			\widetilde{a}_{ij}&:=\frac{1}{\Lambda}h_{ij}+\left(\frac{V_i^{(h)}+W_i^{(h)}}{\Lambda}\right)\left(\frac{V_j^{(h)}+W_j^{(j)}}{\Lambda}\right),\\
			\widetilde{b}_i&:=-\frac{W_i^{(h)}}{\Lambda}-\frac{V_i^{(h)}}{\Lambda},
		\end{split}
	\end{equation}
	where $V_i^{(h)}=h_{ij}V^j$ and $W_i^{(h)}=h_{ij}W^j$.
	
	Next, we will consider the case (II) which we regard as a two steps Zermelo type navigation:
	
	$\underline{\text{Step 1}}$. Consider the Zermelo's navigation with data $(M,h)$ and wind $V$, $\|V\|_h^2<1$ with the solution $F_1=\alpha+\beta$, where
	\begin{equation*}
		\begin{split}
			a_{ij}=\frac{1}{\lambda}h_{ij}+\left(\frac{V_i^{(h)}}{\lambda}\right)\left(\frac{V_j^{(h)}}{\lambda}\right),\ b_i=-\frac{V_i^{(h)}}{\lambda},
		\end{split}
	\end{equation*}
	where $\lambda=1-\|V\|_h^2$, $V_i^{(h)}=h_{ij}V^j$.
	
	$\underline{\text{Step 2}}$. Consider the Zermelo's navigation with data $(M,F_1=\alpha+\beta)$ obtained at step 1, and wind $W$ such that $F_1(-W)<1$, with solution $\widetilde{F}=\widehat{\alpha}+\widehat{\beta}$ (see Proposition \ref{prop_2steps_Zermelo}), where
	\begin{equation}\label{eq_1.3*}
		\begin{split}
			\widehat{a}_{ij}&=\frac{1}{\eta}(a_{ij}-b_ib_j)+\left(\frac{W_i^{(\alpha)}-b_i[1+\beta(W)]}{\eta}\right)\left(\frac{W_j^{(\alpha)}-b_j[1+\beta(W)]}{\eta}\right),\\
			\widehat{b}_i&=-\frac{W_i^{(\alpha)}}{\eta}
		\end{split}
	\end{equation}
	with
	$$
	\eta=[1+\beta(W)]^2-\alpha^2(W), \text{ and } W_i^{(\alpha)}=a_{ij}W^j.
	$$
	
	We will show that $\widetilde{a}_{ij}=\widehat{a}_{ij}$ and $\widetilde{b}_i=\widehat{b}_i$, respectively, for all indices $i,j\in\{1,\dots,n\}$. It is trivial to see that $\Lambda=\lambda-\|W\|_h^2-2<V,W>_h$.
	
	Next, by straightforward computation we get
	$$
	\alpha^2(W)=a_{ij}W^iW^j=\frac{1}{\lambda}\|W\|_h^2+\left(\frac{h(V,W)}{\lambda}\right)^2,\ \beta(W)=-\frac{h(V,W)}{\lambda}.
	$$
	
	It follows that
	$$
	\eta=\left[1-\frac{h(V,W)}{\lambda}\right]^2-\frac{1}{\lambda}\|W\|_h^2-\frac{h^2(V,W)}{\lambda^2}=1-2\frac{h(V,W)}{\lambda}-\frac{1}{\lambda}\|W\|_h^2,
	$$
	we get
	\begin{equation}\label{eq_L1}
		\eta=\frac{\Lambda}{\lambda}.
	\end{equation}
	
	In a similar manner,
	$$
	\frac{W_i^{(\alpha)}-b_i[1+\beta(W)]}{\eta}=\frac{1}{\eta}\left[\frac{h_{ij}W^j}{\lambda}+\frac{V_i^{(h)}W^i}{\lambda}\frac{V_j^{(h)}W^j}{\lambda}+\frac{V_i^{(h)}}{\lambda}\left(1-\frac{h(V,W)}{\lambda}\right)\right],
	$$
	hence we obtain
	$$
	\frac{W_i^{(\alpha)}-b_i(1+\beta(W))}{\eta}=\frac{W_i^{(h)}+V_i^{(h)}}{\Lambda},
	$$
	that is $\widetilde{b}_i=\widehat{b}_i$.
	
	It can be also seen that
	$$
	\frac{1}{\eta}(a_{ij}-b_ib_j)=\frac{1}{\Lambda}h_{ij},
	$$
	hence $\widetilde{a}_{ij}=\widehat{a}_{ij}$ and the identity of formulas \eqref{eq_1.1*} and \eqref{eq_1.3*} is proved. In order to finish the proof we show that the conditions
	
	(i) $\|V+W\|^2_h<1$
	
	and
	
	(ii) $\|V\|_h^2<1$ and $F(-W)<1$
	
	are actually equivalent.
	
	Geometrically speaking, the 2-steps Zermelo's navigation is the rigid translation of $\Sigma_h$ by $V$ followed by the rigid translation of $\Sigma_{F_1} $ by $W$. This is obviously equivalent to the rigid translation of $\Sigma_h$ by $\widetilde{W}=V+W$.
	
	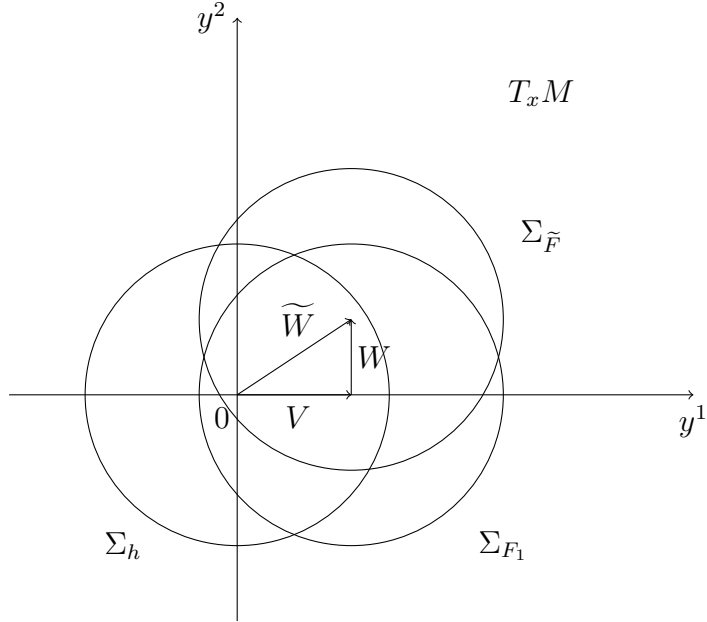
\begin{figure}[H]\label{fig_indicatrix_tilde_V}
		\begin{center}
			\setlength{\unitlength}{1cm}
			\begin{tikzpicture}[scale=1]
				\draw[->](0,3) -- (9,3) node[below]{$y^1$};
				\draw[->](3,0) -- (3,8) node[left]{$y^2$};
				\draw[->](3,3) -- (4.5,3);
				\draw[->](4.5,3) -- (4.5,4);
				\draw[->](3,3) -- (4.5,4);
				\draw(3,3) circle(2);
				\draw(4.5,3) circle(2);
				\draw(4.5,4) circle(2);
				\draw(2.8,2.7) node{$0$};
				\draw(3.8,2.7) node{$V$};
				\draw(3.8,4) node{$\widetilde{W}$};
				\draw(4.8,3.5) node{$W$};
				\draw(1.5,1) node{$\Sigma_h$};
				\draw(7,7) node{$T_xM$};
				\draw(6.5,1) node{$\Sigma_{F_1}$}; 
				\draw(7,5.1) node{$\Sigma_{\widetilde{F}}$};
			\end{tikzpicture}
		\end{center}
		\caption{The $h$-indicatrix, $F_1$-indicatrix and $F$-indicatrix.}
		\label{fig_2}
	\end{figure}
	
	The geometrical meaning of (i) is that the origin $O_x\in T_xM$ is in the interior of the translated indicatrix $\Sigma_{\widetilde{F}}$ (see Figure \ref{fig_2}. On the other hand, the relation in (ii) shows that the origin $O_x$ is in the interior of the translated indicatrix $\Sigma_h$ by $V$ and $\Sigma_{F_1}$ by $W$.
	
	This equivalence can also be checked analytically.
	
	For initial data $(M,h)$ and $V$, we obtain by Zermelo's navigation the Randers metric $F=\alpha+\beta$, where
	\begin{equation*}
		\begin{split}
			a_{ij}=\frac{1}{\lambda}h_{ij}+\left(\frac{V_i}{\lambda}\right)\left(\frac{V_j}{\lambda}\right),\ b_i=-\frac{V_i}{\lambda},
		\end{split}
	\end{equation*}
	with $V_i=h_{ij}V^j$ and $\lambda=1-\|V\|^2_h<1$.
	
	Consider another vector field $W$ and compute
	\begin{equation*}
		\begin{split}
			F(-W)&=\sqrt{\frac{1}{\lambda}\|W\|^2_h+\left(\frac{V_iW^i}{\lambda}\right)^2}+\frac{V_iW^i}{\lambda}\\
			&=\frac{1}{\lambda}\left[\sqrt{\lambda\|W\|_h^2+h^2(V,W)}+h(V,W)\right].
		\end{split}
	\end{equation*}
	
	Let us assume $F(-W)<1$, hence
	$$
	\sqrt{\lambda\|W\|_h^2+h^2(V,W)}+h(V,W)<\lambda,
	$$	
	i.e.
	\begin{equation*}
		\begin{split}
			& \hspace{0.65cm}\lambda\|W\|_h^2+h^2(V,W)<[\lambda-h(V,W)]^2\\
			&\Leftrightarrow \lambda\|W\|_h^2+\cancel{h^2(V,W)}<\lambda^2-2\lambda h(V,W)+\cancel{h^2(V,W)}, \ \lambda >0\\
			&\Leftrightarrow \|W\|_h^2<\lambda - 2 h(V,W)\Leftrightarrow\|W\|_h^2+2h(V,W)+\|V\|_h^2<1\\
			&\Leftrightarrow \|W+V\|_h^2<1.
		\end{split}
	\end{equation*}
	
	Conversely, if $\|V+W\|_h^2<1$, by reversing the computation above, we obtain $F(-W)<1$, provided $\lambda-h(V,W)>0$.
	
	Indeed, observe that $\|V+W\|_h^2<1$ actually implies $\lambda-h(V,W)>0$, because $1-\|V\|_h^2-h(V,W)=1-h(V,V+W)>0\Leftrightarrow h(V,V+W)<1$.
	
	However Cauchy-Schwartz inequality : $h(V,V+W)\leq \|V\|_h\|V+W\|_h<1$ using $\|V\|_h<1$ and $\|V+W\|_h<1$.
\end{proof}

The 2-steps Zermelo's navigation problem discussed above, can be generalized to $k$-steps Zermelo's navigation.

\begin{Remark}
	Let $(M,F)$ be a Finsler space and let $W_0,W_1,\dots,W_{k-1}$ be $k$ linearly independent vector fields on $M$. We consider the following $k$-step Zermelo's navigation problem.
	
	$\underline{\text{Step 0}}$. $F_1$ solution of $(M,F_0,W_0)$ with $F_0(-W_0)<1$, i.e.
	
	\hspace{1cm} Solution of $F_0\left(\frac{y}{F_1}-W_0\right)=1$.
	
	$\underline{\text{Step 1}}$. $F_2$ solution of $(M,F_1,W_1)$ with $F_1(-W_1)<1$, i.e.
	
	\hspace{1cm} Solution of $F_1\left(\frac{y}{F_2}-W_1\right)=1$.
	
	\quad $\vdots$
	
	$\underline{\text{Step k-1}}$. $F_k$ solution of $(M,F_{k-1},W_{k-1})$ with $F_{k-1}(-W_{k-1})<1$, i.e.
	
	\hspace{1cm} Solution of $F_{k-1}\left(\frac{y}{F_k}-W_{k-1}\right)=1$.
	
	Then $F_k$ is the Finsler metric obtained as solution of the Zermelo's navigation problem with data $F_0,\widetilde{W}:=W_0+\dots+W_{k-1}$ with condition $F_0(-\widetilde{W})<1$.
\end{Remark}

\subsection{Geodesics, conjugate and cut loci}

\begin{Proposition}\label{lem_X}
	Let $(M,h)$ be a Riemannian manifold, $V$ a vector field such that $\|V\|_h<1$, and let $F=\alpha+\beta$ be the solution of the Zermelo's navigation with data $(M,h)$ and $V$.
	
	Then $d\beta=0$ if and only if $V$ satisfies the differential equation
	\begin{equation}\label{eq_lem_1}
		d\gamma=d\log\lambda\wedge\gamma,
	\end{equation}
	where $\gamma=V_i(x)dx^i$, $V_i=h_{ij}V^j$, and $\lambda=1-\|V\|_h^2$.
\end{Proposition}

\begin{proof}[Proof of Proposition \ref{lem_X}]
	
	Observe that $b_i=-\frac{V_i}{\lambda}$ is equivalent to $\lambda \beta = -\gamma$, hence $d\lambda \wedge \beta + \lambda d \beta=-d\gamma$ and using $d\beta =0$ we obtain
	$$
	d\lambda \wedge \beta =-d\gamma.
	$$
	
	By using $\beta =-\frac{1}{\lambda}\gamma$ we get \eqref{eq_lem_1} easily. The converse is easy to show taking into account $\lambda\neq 0$. 
\end{proof}

\begin{Remark}
	The equation \eqref{eq_lem_1} can be written in coordinates
	$$
	\left(\frac{\partial V_i}{\partial x^j}-\frac{\partial V_j}{\partial x^i}\right)dx^i\wedge dx^j=\left(\frac{\partial \log \lambda}{\partial x^i}dx^i\right)\wedge (V_jdx^j),
	$$
	that is 
	$$
	\frac{\partial V_i}{\partial x^j}-\frac{\partial V_j}{\partial x^i}=\frac{\partial \log \lambda}{\partial x^i}V_j-\frac{\partial \log \lambda}{\partial x^j}V_i.
	$$

	In the 2-dimensional case, we get the  1st order PDE
	\begin{equation}\label{eq_lem_2}
		\frac{\partial V_1}{\partial x^2}-\frac{\partial V_2}{\partial x^1}=-\frac{1}{\lambda}\left[\frac{\partial h^{ij}}{\partial x^1}V_2-\frac{\partial h^{ij}}{\partial x^2}V_1\right]V_iV_j-\frac{2}{\lambda}V^i\left[\frac{\partial V_i}{\partial x^1}V_2-\frac{\partial V_i}{\partial x^2}V_1\right], \ i,j=1,2.
	\end{equation}
	
	It can easily be seen that in the case of a surface of revolution $h=dr^2+m^2(r)d\theta^2$ the wind $V=A(r)\frac{\partial}{\partial r}$ is a solution of \eqref{eq_lem_1} and of \eqref{eq_lem_2}.
	
\end{Remark}

\begin{theorem}\label{cor: F1 conj points}
	Let $(M,h)$ be a simply connected Riemannian manifold and $V=V^i\frac{\partial}{\partial x^i}$ a vector field on $M$ such that $\|V\|_h<1$, and let $F=\alpha+\beta$ be the Randers metric obtained as the solution of the Zermelo's navigation problem with this data.
	
	If $V$ satisfies the differential relation
	\begin{equation}\label{eq_cor}
		d\eta=d(\log \lambda)\wedge \eta,
	\end{equation}
	where $\eta=V_i(x)dx^i$, $V_i=h_{ij}V^j$, then the followings hold good.
	\begin{enumerate}
		\item There exists a smooth function $f:M\to \R$ such that $\beta=df$.
		\item The Randers metric $F$ is projectively equivalent to $\alpha$, i.e. the geodesics of $(M,F)$ coincide with the geodesics of the Riemannian metric $\alpha$ as non-parametrized curve.
		\item The Finslerian length of  any $C^\infty$ piecewise curve $\gamma:[a,b]\to M$ on $M$ joining the points $p$ and $q$ is given by
		\begin{equation}
			\mathcal L_{F}(\gamma)= \mathcal L_\alpha(\gamma)+f(q)-f(p),
		\end{equation}
		where $ L_\alpha(\gamma)$ is the Riemannian length with respect to $\alpha$ of $\gamma$.
		\item The geodesic $\gamma$ is minimizing with respect to $\alpha$ if and only if it is minimizing with respect to $F$.
		\item For any two points $p$ and $q$  we have
		\begin{equation}
			d_{F}(p,q)=d_\alpha(p,q)+f(q)-f(p),
		\end{equation}
		where $d_\alpha(p,q)$ is the Riemannian distance between $p$ and $q$ with respect to $\alpha$ of $\gamma$.
		\item For an $F$-unit speed geodesic $\gamma$, if we put $p:=\gamma(0)$ and $q:=\gamma(t_0)$, then $q$ is conjugate to $p$ along $\gamma$ with respect to $F$ if and only if  $q$ is conjugate to $p$ along $\gamma$ with respect to $\alpha$.
		\item The cut locus of $p$ with respect to $F$ coincide with the cut locus of $p$ with respect to $\alpha$.
	\end{enumerate}
\end{theorem}
\begin{proof}[Proof of Theorem \ref{cor: F1 conj points}]
	\begin{enumerate}
		\item Using Proposition \ref{lem_X}, it is clear that the differential equation \eqref{eq_cor} is equivalent to $\beta$ closed 1-form, i.e. $d\beta=0$.
		
		On the other hand, since $M$ is simply connected manifold, any closed 1-form is exact, hence in this case \eqref{eq_cor} is equivalent to $\beta=df$.
		\item Follows immediately from the classical result in Finsler geometry that a Randers metric $\alpha+\beta$ is projectively equivalent to its Riemannian part $\alpha$ if and only if $d\beta=0$ (see for instance \cite{BCS}, p.298). 
		\item The length of the curve $\gamma[a,b]\to M$, given by $x^i=x^i(t)$ is defined as
		\begin{equation*}
			\begin{split}
				\cL_{F_1}(\gamma)&=\int_a^bF_1(\gamma(t),\dot{\gamma}(t))dt=
				\int_a^b\alpha(\gamma(t),\dot{\gamma}(t))dt+\int_a^b\beta(\gamma,\dot{\gamma}(t))dt\\&=\cL_\alpha(\gamma)+f(q)-f(p)\\
			\end{split}
		\end{equation*}
		where we use
		$$
		\int_a^b\beta(\gamma(t),\dot{\gamma}(t))dt=\int_a^bdf(\gamma(t),\dot{\gamma}(t))dt=f(\gamma(b))-f(\gamma(a))=f(q)-f(p).
		$$
		\item It follows from 3.
		\item It follows immediately from 2 and 3 (see \cite{SSS} for a detailed discussion on this type of distance).
		\item From (2) we know that $\alpha$ and $F=\alpha+\beta$ are projectively equivalent, i.e. their non-parametrized geodesics coincide as set points. More precisely, if $\gamma:[0,l]\to M$, $\gamma(t)=(x^i(t))$ is an $\alpha$-unit speed geodesic, and $\overline{\gamma}:[0,\widetilde{l}]\to M$, $\overline{\gamma}(s)=(x^i(s))$ is an $F$-unit speed geodesic, then there exists a parameter changing $t=t(s)$, $\frac{dt}{ds}>0$ such that $\gamma(t)=\overline{\gamma}(t(s))$ with the inverse function $s=s(t)$ such that $\overline{\gamma}(s)=\gamma(s(t))$.
		
		\quad Observe that if $q=\gamma(a)$ then $q=\overline{\gamma}(\widetilde{a})$, where $t(\widetilde{a})=a$.
		
		\quad Let us consider a Jacobi field $Y(t)$ along $\gamma$ such that
		\begin{equation*}
			\begin{cases}
				Y(0)=0\vspace{0.2cm} \\
				<Y(t),\frac{d\gamma}{dt}>_{\alpha}=0,
			\end{cases}
		\end{equation*}
		and construct the geodesic variation $\gamma:[0,a]\times (-\varepsilon,\varepsilon)\to M$, $(t,u)\mapsto \gamma(t,u)$ such that
		
		\begin{equation*}
			\begin{cases}
				\gamma(t,0)=\gamma(t)  \vspace{0.2cm} \\
				\frac{\partial \gamma}{\partial u}\Big\vert_{u=0}=Y(t).
			\end{cases}
		\end{equation*}  
		
		\quad Since the variation vector field $\frac{\partial \gamma}{\partial u}\Big\vert_{u=0}$ is Jacobi field it follows that all geodesics $\gamma_u(t)$ in the variation are $\alpha$-geodesics for any $u\in(-\varepsilon,\varepsilon)$. 
		
		\quad Similarly with the case of base manifold, every curve in the variation can be reparametrized as an $F$-geodesic. In other words, for each $u\in(-\varepsilon,\varepsilon)$ it exists a parameter changing $t=t(s,u)$, $\frac{\partial t}{\partial s}>0$ such that 
		$$
		\gamma(t,u)=\overline{\gamma}(t(s,u),u).
		$$
		
		\quad We will compute the variation vector field of the variation $\overline{\gamma}(s,u)$ as follows
		
		$$
		\frac{\partial \overline{\gamma}}{\partial u}(s,u)=\frac{\partial \gamma}{\partial t}\Big\vert_{(t(s,u),u)} \frac{\partial t}{\partial u}(s,u)+\frac{\partial\gamma}{\partial u}\Big\vert_{(t(s,u),u)}.
		$$
		
		\quad If we evaluate this relation for $u=0$ we get
		
		$$
		\frac{\partial \overline{\gamma}}{\partial u}(s,0)=\frac{\partial \gamma}{\partial t}\Big\vert_{(t(s,0),0)} \frac{\partial t}{\partial u}(s,0)+\frac{\partial\gamma}{\partial u}\Big\vert_{(t(s,0),0)},
		$$
		that is
		$$
		\overline{Y}(s)=\frac{\partial \gamma}{\partial t}\Big\vert_{t(s,0),0}\frac{\partial t}{\partial u}\Big\vert_{(s,0)}+Y\Big\vert_{(t(s,0),0)}\in T_{\overline{\gamma}(s)}M\equiv T_{\gamma(t(s))}M.
		$$
		
		For a point $q=\gamma(a)=\overline{\gamma}(\widetilde{a})$ this formula reads
		\begin{equation}\label{eq_proof_cor_6}
			\begin{split}
				\overline{Y}(\widetilde{a})&=\frac{\partial \gamma}{\partial t}\Big\vert_{\widetilde{a}}\frac{\partial t}{\partial u}\Big\vert_{(\widetilde{a},0)}+Y(t(\widetilde{a}))\\
				&=\frac{d\gamma}{dt}\Big\vert_a\frac{\partial t}{\partial u}\Big\vert_{(\widetilde{a},0)}+Y(a)\in T_{\overline{\gamma}(\widetilde{a})}M\equiv T_{\gamma(a)}M,
			\end{split}
		\end{equation}
		i.e. the Jacobi field $\overline{Y}(\widetilde{a})$ is linear combination of the tangent vector $\frac{\partial \gamma}{\partial t}(a)$ and $Y(a)$.
		
		Let us assume $q=\overline{\gamma}(\widetilde{a})$ is conjugate point to $p$ along the $F$-geodesic $\overline{\gamma}$, i.e. $\overline{Y}(\widetilde{a})=0$. It results $\frac{d\gamma}{dt}(a)$ cannot be linear independent, hence $Y(a)=0$, i.e. $q=\gamma(a)$ is conjugate to $p$ along the $\alpha$-geodesic $\gamma$.
		
		Conversely, if $q=\gamma(a)$ is conjugate to $p$ along the $\alpha$-geodesic $\gamma$ then \eqref{eq_proof_cor_6} can be written as
		$$
		Y(a)=\overline{Y}(s(a))-\frac{d\overline{\gamma}}{ds}(s(a))\frac{ds}{dt}\frac{dt}{du}
		$$
		and the conclusion follows from the same linearly independence argument as above.
		
		\item Observe that $Cut_\alpha(p)\neq\emptyset \Leftrightarrow Cut_F(p)\neq \emptyset$.
		
		Indeed, if $Cut_\alpha(p)= \emptyset$ all $\alpha$-geodesics from $p$ are globally minimizing. Assume $q\in Cut_F(p)$ and we can consider $q$ end point of $Cut_F(p)$, i.e. $q$ must be $F$-conjugate to $p$ along the geodesic $\sigma(s)$ from $p$ to $q$. This implies the corresponding point on $\sigma(t)$ is conjugate to $p$, this is a contradiction.
		
		Converse argument is identical.
		
		Let us assume $Cut_\alpha(p)$ and $Cut_F(q)$ are not empty sets.
		
		If $q\in Cut_\alpha(p)$ then we have two cases:
		\begin{itemize}
			\item[(i)] $q$ is an end point of $Cut_\alpha(p)$, i.e. it is conjugate to $p$ along a minimizing geodesic $\gamma$ from $p$ to $q$. Therefore $q$ is closes conjugate to $p$ along the $F$-geodesic $\overline{\gamma}$ which is the reparametrization of $\gamma$ (see 6).
			\item[(ii)] $q$ is an interior point of $Cut_\alpha(p)$. Since the set of points in $Cut_\alpha(p)$ founded at the intersection of exactly minimizing two geodesics of same length is dense in the closed set $Cut_\alpha(p)$ it is enough to consider this kind of cut points. In the case $q\in Cut_\alpha(p)$ such that there are 2 $\alpha$-geodesics $\gamma_1$, $\gamma_2$ of same length from $p$ to $q=\gamma_1(a)=\gamma_2(a)$, then from (4) it is clear that the point $q=\overline{\gamma}_1(\widetilde{a})=\overline{\gamma}_2(\widetilde{a})$ has the same property with respect to $F$.
			
			Hence $Cut_\alpha(p)\subset Cut_F(p)$. This inverse conclusion follows from the same argument as above by changing roles of $\alpha$ with $F$.
		\end{itemize}
	\end{enumerate}
\end{proof}

\begin{Remark}
	See \cite{INS} for a more general case.
\end{Remark}

We recall the following well-known result for later use.
\begin{Lemma}\label{lem_Y0}\textnormal{(\cite{HS},\cite{MHSS})}
	Let $F=\alpha+\beta$ be the solution of Zermelo's navigation problem with navigation data $(h,V)$, $\|V\|_h<1$.
	
	Then the Legendre dual of $F$ is Hamiltonian function $F^*=\alpha^*+\beta^*$ where ${\alpha^*}^2=h^{ij}(x)p_ip_j$ and $\beta^*=V^i(x)p_i$. Here $(x,p)$ are the canonical coordinates of the cotangent bundle $T^*M$.
	
	Moreover, $g_{ij}(x,y){g^*}^{ik}(x,p)=\delta_{j}^k$, where $F^2(x,y)=g_{ij}(x,y)y^iy^j$ and ${F^*}^2(x,p)={g^*}^{ij}(x,p)p_ip_j$.
\end{Lemma}

The following result is similar to the Riemannian counterpart and we give it here with proof.

We recall that a smooth vector field $X$ on a Finsler manifold $(M,F)$ is called {\it Killing field} if every local one-parameter transformation group $\{\varphi_t\}$ of $M$ generated by $X$ consists of local isometries.
It is clear from our construction above that $W$ is Killing field on the surface of revolution $(M,F)$. 
We also have
\begin{Proposition}\label{prop_Killing}
	Let $(M,F)$ be a Finsler manifold (any dimension) with local coordinates $(x^i,y^i)\in TM$ and $X=X^i(x)\frac{\partial}{\partial x^i}$ a vector field on $M$. The following formulas are equivalent
	\begin{enumerate}
		\item[(i)] $X$ is Killing field for $(M,F)$;
		\item[(ii)] $\cL_{\widehat{X}}F=0$, where $\widehat{X}:=X^i\frac{\partial}{\partial x^i}+y^j\frac{\partial X^i}{\partial x^j}\frac{\partial}{\partial y^i}$ is the canonical lift of $X$ to $TM$;
		\item[(iii)]
		$$
		\frac{\partial g_{ij}}{\partial x^p}X^p+g_{pj}\frac{\partial X^p}{\partial x^i}+g_{ip}\frac{\partial x^p}{\partial x^j}+2C_{ijp}\frac{\partial x^p}{\partial x^q}y^q=0;
		$$
		\item[(iv)] $X_{i|j}+X_{j|i}+2C_{ij}^pX_{p|q}y^q=0$, where `` $|$\,'' is the $h$-covariant derivative with respect to the Chern connection.
	\end{enumerate}
\end{Proposition}

\begin{Lemma}\label{lem_Y1}
	With the notation in Lemma \ref{lem_Y0}, the vector field $W=W^i(x)\frac{\partial}{\partial x^i}$ on $M$ is Killing field with respect to $F$ if and only if 
	$$
	\{F^*,W^*\}=0,
	$$
	where $W^*=W^i(x)p_i$ and $\{\cdot,\cdot\}$ is the Poincar\'e bracket.
\end{Lemma}

\begin{proof}[Proof of Lemma \ref{lem_Y1}]
	Recall that $W$ is Killing field of $(M,F)$ if and only if every local one-parameter transformation group $\{\varphi_t\}$ of $M$ generated by $W$ consists of local isometries.
	
	A straight forward computation shows that $W$ is Killing on $(M,F)$ if and only if $\cL_{\widehat{W}}F=0$, where
	$$
	\widehat{W}=W^i\frac{\partial}{\partial x^i}+y^j\frac{\partial W^i}{\partial x^j}\frac{\partial}{\partial y^i}
	$$
	is the canonical lift of $W$ to $TM$. In local coordinates this is equivalent to
	\begin{equation}\label{eq_lem_Y1_1}
		\frac{\partial g_{ij}}{\partial x^p}W^p+g_{pj}\frac{\partial W^p}{\partial x^i}+g_{ip}\frac{\partial W^p}{\partial x^j}+2C_{ijp}\frac{\partial W^p}{\partial x^q}y^q=0.
	\end{equation}
	
	Since the left hand side is 0-homogeneous in the $y$-variable, this relation is actually equivalent to the contracted relation by $y^iy^j$, i.e. \eqref{eq_lem_Y1_1} is equivalent to
	$$
	\left(\frac{\partial g_{ij}}{\partial x^p}W^p+g_{pj}\frac{\partial W^p}{\partial x^i}+g_{ip}\frac{\partial W^p}{\partial x^j}\right)y^iy^j=0,
	$$
	where we use $C_{ijk}y^i=0$. We get the equivalent relation
	\begin{equation}\label{eq_lem_Y1_2}
		\frac{\partial g_{ij}}{\partial x^p}W^py^iy^j+2g_{pj}\frac{\partial W^p}{\partial x^i}y^iy^j=0.
	\end{equation}
	
	Observe that $g_{ij}{g^*}^{jk}=\delta_i^k$ is equivalent to $\frac{\partial g_{ij}}{\partial x^p}{g^*}^{ik}=-g_{ij}\frac{\partial {g^*}^{ik}}{\partial x^p}$, hence \eqref{eq_lem_Y1_2} reads
	$$
	\frac{\partial g_{ij}}{\partial x^p}W^p\left({g^*}^{ik}p_k\right)\left({g^*}^{jl}p_l\right)+2g_{pj}\frac{\partial W^p}{\partial x^i}\left({g^*}^{ik}p_k\right)\left({g^*}^{jl}p_l\right)=0
	$$
	and from here
	$$
	-g_{ij}\frac{\partial {g^*}^{ik}}{\partial x^p}W^pp_k{g^*}^{jl}p_l+2g_{pj}\frac{\partial W^p}{\partial x^i}\left({g^*}^{ik}p_k\right)\left({g^*}^{jl}p_l\right)=0.
	$$
	
	We finally obtain 
	\begin{equation}\label{eq_lem_Y1_3}
		-\frac{\partial {g^*}^{ik}}{\partial x^p}W^pp_ip_k+2{g^*}^{jk}\frac{\partial W^i}{\partial x^j}p_ip_k=0.
	\end{equation}
	
	On the other hand, we compute
	
	\begin{equation*}
		\begin{split}
			\{{F^*}^2,W^*\}&=\{{g^*}^{ij}p_ip_j,W^sp_s\}\\
			&=\frac{\partial ({g^*}^{ij}p_ip_j)}{\partial p_k}\frac{\partial (W^sp_s)}{\partial x^k}-\frac{\partial ({g^*}^{ij}p_ip_j)}{\partial x^k}\frac{\partial (W^sp_s)}{\partial p_k}\\
			&=\left(\frac{\partial {g^*}^{ij}}{\partial p_k}p_ip_j+2{g^*}^{ik}p_i\right)\frac{\partial W^s}{\partial x^k}p_s-\frac{\partial {g^*}^{ij}}{\partial x^k}p_ip_jW^k\\
			&=2{g^*}^{ik}\frac{\partial W^s}{\partial x^k}p_ip_s-\frac{\partial {g^*}^{ij}}{\partial x^k}W^kp_ip_j
		\end{split}
	\end{equation*}
	which is the same with \eqref{eq_lem_Y1_3}. Here we have used the 0-homogeneity of ${g^*}^{ij}(x,p)$ with respect to $p$.
	
	We also observe that for any functions $f,\ g:T^*M\to \R$ we have $\{f^2,g\}=2f\{f,g\}$.
	
	Therefore, the following are equivalent
	\begin{itemize}
		\item[(i)] $W$ is Killing field on $(M,F)$;
		\item[(ii)] $\cL_{\widehat{W}}F=0$;
		\item[(iii)]  formula \eqref{eq_lem_Y1_2}
		\item[(iv)]  formula \eqref{eq_lem_Y1_3}
		\item[(v)] $\{{F^*}^2,W^*\}=0$
		\item[(vi)] $\{F^*,W^*\}=0$
	\end{itemize}
	and the lemma is proved.
\end{proof}

\begin{Proposition}[\cite{FM}] \label{lem_FM}
	Let $(M,F)$ be a Finsler manifold and $W=W^i(x)\frac{\partial}{\partial x^i}$ a Killing filed on $(M,F)$ with $F(-W)<1$. If we denote by $\widetilde{F}$ the solution of the Zermelo's navigation problem with data $(F,W)$, then the following are true
	\begin{enumerate}
		\item The $\widetilde{F}$-unit speed geodesics $\cP(t)$ can be written as 
		$$
		\cP(t)=\varphi(t,\rho(t)),
		$$ 
		where $\varphi_t$ is the 1-parameter flow of $W$ and $\rho$ is an $F$-unit speed geodesic.
		\item For any Jacobi field $J(t)$ along $\rho(t)$ such that $g_{\dot{\rho}(t)}(\dot{\rho}(t),J(t))=0$, the vector field $\widetilde{J}(t):=\varphi_{t*}(J(t))$ is a Jacobi field along $\cP$ and $\widetilde{g}_{\dot{\cP}(t)}(\dot{\cP}(t),\widetilde{J}(t))=0$.
		\item For any $x\in M$ and any flag $(y,V)$ with flag pole $y\in T_xM$ and transverse edge $V\in T_xM$, the flag curvatures $K$ and $\widetilde{K}$ of $F$ and $\widetilde{F}$, respectively, are related by
		$$
		{K}(x,y,V)=\widetilde{K}(x,y+W,V)
		$$ 
		provided $y+W$ and $V$ are linearly independent.
	\end{enumerate}
\end{Proposition}

In the 2-dimensional case, since any Finsler surface is of scalar flag curvature, we get

\begin{Corollary}
	In the two-dimensional case, with the notation in Proposition \ref{lem_FM}, the Gauss curvature $K$ and $\widetilde{K}$ of $F$ and $\widetilde{F}$ are related by $K
	(x,y)= \widetilde{K}(x,y+W)$, for any $(x,y)\in TM$.
\end{Corollary}

\begin{Lemma}\label{lem_SC}
	Let $(M,F)$ be a (forward) complete Finsler manifold, and let $W$ be a Killing field with respect to $F$. Then $W$ is a complete vector field on $M$, i.e. for any $x\in M$ the flow $\varphi_x(t)$ is defined for any $t$.
\end{Lemma}

\begin{proof}[Proof of Lemma \ref{lem_SC}]
	Since $W$ is Killing field, it is clear that its flow $\varphi$ preserves the Finsler metric $F$ and the field $W$. In other words, for any $p\in M$, the curve $\alpha:(a,b)\to M$, $\alpha(t)=\varphi_x(t)$ has constant speed. 
	
	Indeed, it is trivial to see that
	\begin{equation*}
		\begin{split}
			\frac{d}{dt}F(\gamma(t),W\gamma(t))&=\frac{\partial F}{\partial x^i}\frac{d\gamma^i}{dt}+\frac{\partial F}{\partial y^i}\frac{\partial W^i}{\partial x^k}\frac{d\gamma^k}{dt}\\
			&=\frac{\partial F}{\partial x^i}W^i+\frac{\partial F}{\partial y^i}\frac{\partial W^i}{\partial x^k}W^k=\cL_WF(W)=0.
		\end{split}
	\end{equation*}
	
	It means that the $F$-length of $\alpha$ is $b-a$, i.e. finite, hence by completeness it can be extended to a compact domain $[a,b],$ and therefore $\alpha$ is defined on whole $\R$. It results $W$ is complete.
\end{proof}

\begin{theorem}\label{lem_S}
	Let $(M,F)$ be a Finsler manifold (not necessary Randers) and $W=W^i(x)\frac{\partial}{\partial x^i}$ a Killing field for $F$, with $F(-W)<1$.
	
	If $\widetilde{F}$ is the solution of the Zermelo's navigation problem with data $(M,F)$ with the wind $W$ then the followings hold good:
	\begin{enumerate}
		\item[(i)] The point $\cP(l)$ is $\widetilde{F}$-conjugate to $\cP(0)$ along the $\widetilde{F}$-geodesic $\cP(t)=\varphi(t,\rho(t))$ if and only if the corresponding point $\rho(l)=\varphi(-l,\cP(l))$ is the $F$-conjugate point to $\cP(0)=\rho(0)$ along $\rho$.
		\item[(ii)] $(M,F)$ is (forward) complete if and only if $(M,\widetilde{F})$ is (forward) complete.
		\item[(iii)] If $\rho$ is a $F$-global minimizing geodesic from $p=\rho(0)$ to a point $\widehat{q}=\rho(l)$, then $\cP(t)=\varphi(t,\rho(t))$ is an $\widetilde{F}$-global minimizing geodesic from $p=\cP(0)$ to $q=\cP(l)$, where $l=d_F(p,\widehat{q})$.
		\item[(iv)] If $\widehat{q}\in cut_F(p)$ is a $F$-cut point of $p$, then $q=\varphi(l,\widehat{q})\in cut_{\widetilde{F}}(p)$, i.e. it is a $\widetilde{F}$-cut point of $p$, where $l=d_F(p,\widehat{q})$.
	\end{enumerate}
\end{theorem}

\begin{proof}[Proof of Theorem \ref{lem_S}]
	\begin{enumerate}
		\item[(i)] Since $\varphi_t(\cdot)$ is a diffeomorphism on $M$ (see Lemma \ref{lem_SC}), it is clear that its tangent map $\varphi_{t*}$ is a regular linear mapping (Jacobian of $\varphi_t$ is non-vanishing). Then Lemma \ref{lem_FM} shows that $\widetilde{J}$ vanishes if and only if $J$ vanishes, and the conclusion follow easily.
		\item[(ii)] Let us denote by $\exp_p:T_pM\to M$ and $\widetilde{\exp}_p:T_pM\to M$ the exponential maps of $F$ and $\widetilde{F}$, respectively. Then $\cP(t)=\varphi(t,\rho(t))$ implies
		\begin{equation}\label{eq_proof_lem_S_1}
			\widetilde{\exp}_p(ty)=\varphi_t\circ \exp_p(t[y-W(p)]).
		\end{equation}
		
		If $(M,F)$ is complete, Hopf-Rinow theorem for Finsler manifolds implies that for any $p\in M$, the exponential map $\exp_p$ is defined on all of $M$. Taking into account Lemma \ref{lem_SC}, from \eqref{eq_proof_lem_S_1} it follows $\widetilde{\exp}_p$ is defined on all of $T_pM$, and again by Hopf-Rinow theorem we obtain that $\widetilde{F}$ is complete. The converse proof is similar.
		
		\item[(iii)] Firstly observe that $l=d_F(p,\widehat{q})=d_F(p,q)$, since $\widehat{q}=\rho(l)=\varphi(-l,\cP(l))=\varphi(-l,q)$ and $q=\cP(l)=\varphi(l,\rho(l))=\varphi(l,\widehat{q})$.
		
		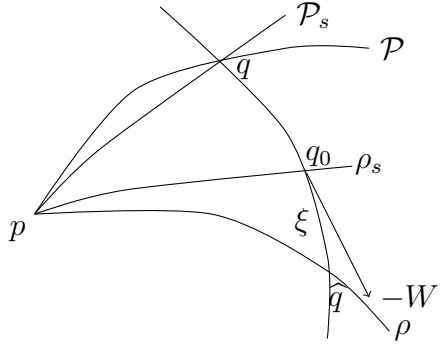
\begin{figure}[H]
			\begin{center}
				\setlength{\unitlength}{1cm}
				\begin{tikzpicture}[scale=1.1]
					\draw[->](3.25,0.5) -- (4,-1) node[below,right]{$-W$};
					\draw(-0.2,-0.2) node{$p$};
					\draw plot [smooth] coordinates {(0,0) (1.2,1.5) (3,2) (4,2)} node[right]{$\cP$};
					\draw plot [smooth,rotate=-45] coordinates {(0,0) (1.5,1.5) (3,2) (4,2)};
					\draw plot [smooth,rotate=-90, yshift=1.5cm, xshift=-2.5cm] coordinates {(0,0) (1.5,1.5) (3,2) (4,2)};
					\draw plot [smooth] coordinates {(0,0) (0.8,0.8) (3,2.4)} node[right]{$\cP_s$};
					\draw plot [smooth,rotate=-30] coordinates {(0,0) (0.8,0.8) (3,2.4)};
					\draw(2.5,1.75)node{$q$};
					\draw(3.4,0.7)node{$q_0$};
					\draw(4,0.6)node{$\rho_s$};
					\draw(3.2,-0.1)node{$\xi$};
					\draw(3.6,-1)node{$\widehat{q}$};
					\draw(4.4,-1.4)node{$\rho$};
				\end{tikzpicture}
			\end{center}
			\caption{Riemannian and Finsler geodesics in Zermelo's navigation problem.}
			\label{fig_3}
		\end{figure}
		
		We will proof this statement by contradiction (see Figure \ref{fig_3}).
		
		For this, let us assume that, even though $\rho$ is globally minimizing, the flow-corresponding geodesic $\cP$ from $p$ to $q$ is not minimizing anymore. In other words, there must exist a shorter minimizing geodesic $\cP_s:[0,l_0]\to M$ from $p$ to $q=\cP_s(l_0)$ such that $d_{\widetilde{F}}(p,q)=l_0<l$. (We use the subscript $s$ for short).
		
		We consider next, the $F$-geodesic $\rho_s:[0,l_0]\to M$ obtained from $\cP$ by flow deviation, i.e. $\rho_s(t)=\varphi(-t,\cP_s(t))$, and denote $q_0=\rho_s(l_0)=\varphi(-l_0,\cP(l_0))$. Then, triangle inequality in $pq_0\widehat{q}$ shows that
		$$
		\cL_F(\rho)\leq \cL_F(\rho_s)+\cL_F(\xi),
		$$
		where we denote by $\xi$ the flow orbit from $W$ through $q$m oriented from $q_0$ to $\widehat{q}$. In other words $\dot{\xi}(t)=-W$, and using the hypothesis $F(-W)<1$, it follows
		\begin{equation}\label{eq_proof_lem_S_2}
			\cL_F(\xi)=\int_a^bF(-W)dt<b-a=\cL_F(\rho)-\cL_F(\rho_s).
		\end{equation}
		
		By comparing relations \eqref{eq_proof_lem_S_1} with \eqref{eq_proof_lem_S_2} it can be seen that this is a contradiction, hence $\cP$ must be globally minimizing.
		\item[(iv)] It follows from (iii) and the definition of cut locus.
	\end{enumerate}
\end{proof}

\begin{Remark}
	Observe that statement (iii) and (iv) are not necessary and sufficient conditions, Indeed, from the proof of (iii) it is clear that for proving $\rho$ global minimizer implies $\cP$ global minimizer we have used condition $F(-W)<1$, which is equivalent to the fact that $\widetilde{F}$-indicatrix includes the origin of $T_pM$, a necessary condition for $\widetilde{F}$ to be positive defined (see Remark \ref{rem_F_positive}).
	
	Likewise, if we want to show that $\cP$ global minimizer implies $\rho$ global minimizer, we need $F(W)<1$, that is, the indicatrix $\Sigma_F$ translated by $-W$ must also include the origin, i.e. the metric $\widetilde{F}_2$ defined by $F(y+\widetilde{F}_2W)=\widetilde{F}_2$, with the indicatrix $\Sigma_{\widetilde{F}_2}=\Sigma_F-W$ is also a positive defined Finsler metric.
	
	In conclusion if we assume $F(-W)<1$ and $F(W)<1$ then the statements (iii) and (iv) in Theorem \ref{lem_S} can be written with ``if and only if''.
\end{Remark}

\begin{Lemma}\label{lem_Y2}
	Let $F=\alpha+\beta$ be the solution of Zermelo's navigation problem with navigation data $(h,V)$.
	
	Then a vector field $W$ on $M$ is Killing with respect to $F=\alpha+\beta$ if $W$ is Killing with respect to $h$ and $[V,W]=0$, where  $[\cdot,\cdot]$ is the Lie bracket.
\end{Lemma}

\begin{proof}[Proof of Lemma \ref{lem_Y2}]
	The proof is immediate from Lemmas \ref{lem_Y0} and \ref{lem_Y1}, Indeed, $W$ is Killing on $(M,F)$ if and only if $\{F^*,W^*\}=0$, hence $\{\alpha^*+\beta^*,W^*\}=\{\alpha^*,W^*\}+\{\beta^*,W^*\}=0$. If $\{\alpha^*,W^*\}=0$, i.e. $W$ is Killing with respect to $h$ and $\{\beta^*,W^*\}=\{V^*,W^*\}=0$. Let us observe that $\{V^*,W^*\}=0$ is actually equivalent to $[V,W]=0$. Geometrically, this means that the flows of $V$ and $W$ commute locally, then the conclusion follows.
\end{proof}

Observe that in local coordinates the conditions in Lemma \ref{lem_Y2} reads
\begin{equation*}
	\begin{cases}
		W_{i:j}+W_{j:i}=0 \vspace{0.2cm} \\
		\sum_{i=1}^n\left(\frac{\partial W^k}{\partial x^i}V^i-\frac{\partial V^k}{\partial x^i}W^i\right)=0,
	\end{cases}
\end{equation*}
where $:$ is the covariant derivative with respect to the Levi-Civita connection of $h$.

\begin{theorem}\label{thm: F cut points}
	Let $(M,h)$ be a simply connected Riemannian manifold and $V=V^i\frac{\partial}{\partial x^i}$, $W=W^i\frac{\partial}{\partial x^i}$ vector fields on $M$ such that
	\begin{enumerate}
		\item[(i)] $V$ satisfies the differential relation
		\begin{equation}\label{eq_d_eta}
			d\eta=d(\log \lambda)\wedge \eta,
		\end{equation}
		where $\eta=V_i(x)dx^i$, $V_i=h_{ij}V^j$;
		\item[(ii)] $W$ is Killing with respect to $h$ and $\{V^*,W^*\}=0$, where $V^*=V^ip_i$ and $W^*=W^ip_i$. 
	\end{enumerate}
	Then
	\begin{enumerate}
		\item[(i)] The $\widetilde{F}$-unit speed geodesics $\cP(t)$ are given by
		$$
		\cP(t)=\varphi(t,\sigma(t)),
		$$
		where $\varphi$ is the flow of $W$ and $\sigma(t)$ is an $F_1$-unit speed geodesic. 
		
		Equivalently,
		$$
		\cP(t)=\varphi(t,\gamma(s(t))),
		$$
		where $\gamma(s)$ is an $\alpha$-unit speed geodesic and $s=s(t)$ is the parameter change $t=\int_0^sF_1\left(\rho(\tau),\frac{d\rho}{d\tau}\right)d\tau$.    
		\item[(ii)] The point $\cP(l)$ is conjugate to $\cP(0)=p$ along the $\widetilde{F}-geodesic$ $\cP(t)$ if and only if the corresponding point $\widehat{q}=\rho(l)=\varphi(-l,\cP(l))$ on the $\widetilde{F}$-geodesic $\rho$ is conjugate to $p$, or equivalently, $\widehat{q}$ is conjugate to $p$ along the $\alpha$-geodesic from $p$ to $\widehat{q}$.
		\item[(iii)] If $\widehat{q}\in Cut_\alpha(p)$ then $q=\varphi(l,\widehat{q})\in Cut_{\widetilde{F}}(p)$, where $l=d_{\widetilde{F}}(p,\widehat{q})=d_F(p,\widehat{q})+f(\widehat{q})-f(p)$.
		
	\end{enumerate}
\end{theorem}

\begin{proof}[Proof of Theorem \ref{thm: F cut points}]
	All statements follows immediately by combining Theorem \ref{cor: F1 conj points} with Theorem \ref{lem_S}.
\end{proof}

\begin{Remark}
	Informally, we may say that the cut locus of $p$ with respect to $F$ is the $W$-flow deformation of the cut locus of $p$ with respect to $F_1$, that is, the the $W$-flow deformation of the cut locus of $p$ with respect to $\alpha$, due to Theorem \ref{cor: F1 conj points}, 7. 
\end{Remark}


\section{Surfaces of revolution}\label{sec:Surf of revol}

\subsection{Finsler surfaces of revolution}

Let $(M,F)$ be a (forward) complete oriented Finsler surface, and $W$ a vector field on $M$, whose one-parameter group of transformations $\{\varphi_t:t\in I\}$ consists of $F$-isometries, i.e.
\begin{equation*}
	F(\varphi_t(x),\varphi_{t,x}(y))=F(x,y),\quad\text{for all}\ (x,y)\in TM\ \text{and any} \ t\in\R.
\end{equation*}

This is equivalent with
\begin{equation*}
	d_F(\varphi_t(q_1),\varphi_t(q_2))=d_F(q_1,q_2),
\end{equation*}
for any $q_1,q_2\in M$ and any given $t$, where $d_F$ is the Finslerian distance on $M$. If $\varphi_t$ is not the identity map, then it is known that $W$ must have at most two zeros on $M$.

We assume hereafter that $W$ has no zeros, hence from Poincar\'e-Hopf theorem it follows that $M$ is a surface homeomorphic to a plane, a cylinder or a torus. Furthermore, we assume that $M$ is the topological cylinder $\Sph^1\times\R$.

By definition it follows that, at any $x\in M\setminus\{p\}$, $W_x$ is tangent to the curve $\varphi_x(t)$ at the point $x=\varphi_x(0)$. The set of points $Orb_W(x):=\{\varphi_t(x):t\in\R\}$ is called the orbit of $W$ through $x$, or a {\it parallel circle} 
and it can be seen that the period $\tau(x):=\min\{t>0 : \varphi_t(x)=x\}$ is constant for a fixed $x\in M$. 

\begin{Definition}\label{def: Finsler surf of revol}
	A (forward) complete oriented Finsler surface $(M,F)$ homeomorphic to $\Sph^1\times \R$,  with a vector field $W$ that has no zero points, 
	is called a {\it Finsler cylinder of revolution},
	and $\varphi_t$ a {\it rotation} on $M$.
\end{Definition}

It is clear from our construction above that $W$ is Killing field on the surface of revolution $(M,F)$. 

\subsection{The Riemannian case}

The simplest case is when the Finsler norm $F$ is actually a Riemannian one.

A {\it Riemannian cylinder of revolution} $(M,h)$ is a complete Riemannian manifold  $M=\Sph^1\times \R=\{(r,\theta):r\in\R,\ \theta\in[0,2\pi)\}$ with a warped product metric
\begin{equation}\label{eq_Riemannian_metric_h}
	h=dr^2+m^2(r)d\theta^2.
\end{equation}
of the real line $(\R,dr^2)$ and the unit circle $(\Sph^1,d\theta^2)$.

Suppose that the warping function $m$ is a positive-valued even function.

Recall that the equations of an $h$-unit speed geodesic $\gamma(s):=(r(s),\theta(s))$ of $(M,h)$ are
\begin{equation}\label{eq 4}
	\begin{cases}
		\frac{d^2r}{ds^2}-mm'\left(\frac{d\theta}{ds}\right)^2=0\vspace{0.2cm} \\
		\frac{d^2\theta}{ds^2}+2\frac{m'}{m}\frac{dr}{ds}\frac{d\theta}{ds}=0
	\end{cases},
\end{equation}
with the unit speed parametrization condition 
\begin{equation}\label{eq 1}
	\left(\frac{dr}{ds}\right)^2+m^2\left(\frac{d\theta}{ds}\right)^2 =1.
\end{equation}

It follows that every profile curve $\{\theta=\theta_0\}$, or {\it meridian}, is an $h$-geodesic, and that a parallel $\{r=r_0\}$ is geodesic if and only if
$m'(r_0)=0$, where $\theta_0\in [0,2\pi)$ and $r_0\in \R$ are constants.
It is clear that two meridians do not intersect on $M$ and for a point $p\in M$, the meridian through $p$ does not contain any cut points of $p$, that is, this meridian is a ray through $p$ and hence $d_h(\gamma(0),\gamma(s))=s$, for all $s\geq 0$.

We observe that  (\ref{eq 4}) implies
\begin{equation}\label{h-prime integral}
	\frac{d\theta(s)}{ds}m^2(r(s)) = \nu, \quad \nu \text{ is constant},
\end{equation}
that is the quantity $\frac{d\theta}{ds}m^2$ is conserved along the $h$-geodesics. 

\begin{figure}[H]
	\begin{center}
		\setlength{\unitlength}{1.1cm}
		\begin{picture}(7,7)
			
			\put(3.5,0){\vector(0,1){7}}
			\put(3.5,3.5){\vector(1,0){3.5}}
			\put(5,5){\vector(-1,-1){3}}
			
			\qbezier(4.5,6.5)(4,5.625)(4.5,4.75)
			\qbezier(4.5,4.75)(5,4)(5,3.5)
			
			\qbezier(2.5,6.5)(3,5.625)(2.5,4.75)
			\qbezier(2.5,4.75)(2,4)(2,3.5)
			
			\qbezier(2.5,0.5)(3,1.375)(2.5,2.25)
			\qbezier(2.5,2.25)(2,3)(2,3.5)
			
			\qbezier(4.5,0.5)(4,1.375)(4.5,2.25)
			\qbezier(4.5,2.25)(5,3)(5,3.5)
			
			
			\qbezier(4,0.5)(3.75,1.375)(4,2.25)
			\qbezier(4,2.25)(4.25,3)(4.25,3.5)
			
			\qbezier(4,6.5)(3.75,5.625)(4,4.75)
			\qbezier(4,4.75)(4.25,4)(4.25,3.5)
			
			
			\qbezier(4.25,5.5)(4.25,5.3)(3.5,5.3)
			\qbezier(3.5,5.3)(2.75,5.3)(2.75,5.5)
			\qbezier(4.25,5.5)(4.25,5.7)(3.5,5.7)
			\qbezier(3.5,5.7)(2.75,5.7)(2.75,5.5)
			
			\qbezier(4.25,1.5)(4.25,1.3)(3.5,1.3)
			\qbezier(3.5,1.3)(2.75,1.3)(2.75,1.5)
			\qbezier(4.25,1.5)(4.25,1.7)(3.5,1.7)
			\qbezier(3.5,1.7)(2.75,1.7)(2.75,1.5)
			
			\put(4.2,3.05){\vector(1,1){1}}
			\qbezier(3.5,2.5)(4.5,3)(4.8,4.25)
			
			\put(4.2,3.05){\vector(0.1,1){0.3}}
			
			\qbezier(4.25,3.3)(4.3,3.4)(4.45,3.3)
			
			\qbezier(3.25,6.7)(3.5,6.6)(3.75,6.7)
			\put(3.75,6.7){\vector(1,1){0.1}}
			
			\put(3.5,7.2){$z$}
			\put(4.55,6){$\frac{\partial}{\partial r}$}
			\put(5.25,4){$\dot{\gamma}$}
			\put(6.8,3.1){$x$}
			\put(4.4,3){$\phi$}
			\put(3.6,2.3){$\gamma$}
			\put(5,0.95){$0-meridian$}
			\put(0.4,0.95){$\pi-meridian$}
			\put(4,0){$\theta_0-meridian$}
			\put(2,1.7){$y$}
			\put(0,5){$parallels$}
			
			\put(1.2,5){\vector(1,0.35){1.5}}
			\put(1.2,5){\vector(0.4,-1){1.4}}
			
			\put(2.3,1){\vector(1,0){0.4}}
			\put(4.8,1){\vector(-1,0){0.4}}
			\put(4.2,0.2){\vector(-0.4,1){0.1}}
			
		\end{picture}		
	\end{center}	
	\caption{The angle $\phi$ between $\dot{\gamma}$ and a meridian for a cylinder of revolution.}
	\label{fig_4}
\end{figure}
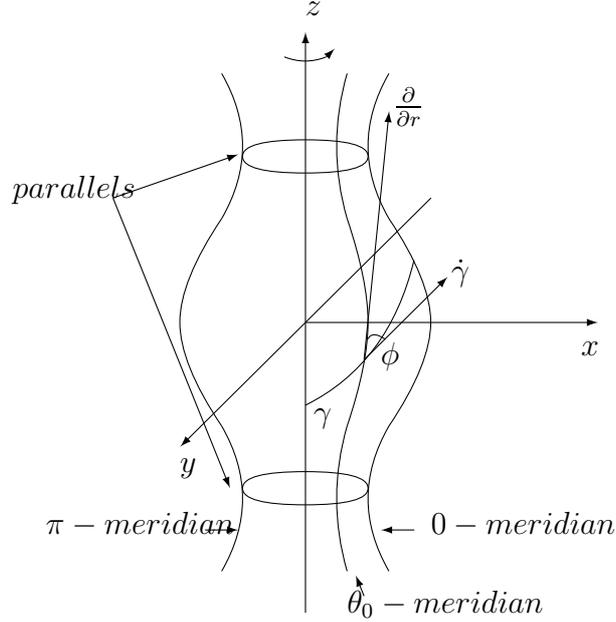

If $\gamma(s)=(r(s),\theta(s))$ is a geodesic on the surface of revolution $(M,h)$, then the angle $\phi(s)$ between $\dot{\gamma}$ and the profile curve passing through a point $\gamma(s)$ satisfy Clairaut relation $m(r(s))\sin\phi(s)=\nu$.

The constant $\nu$ is called the {\it Clairaut constant} (see Figure \ref{fig_4}).

We recall the Theorem of cut locus on cylinder of revolution from \cite{C1}
\begin{theorem}\label{thm_Riemannian_cut_locus}
	Let $(M,h)$ is a cylinder of revolution with the warping function $m:\R\to \R$ is a positive valued even
	function, and the Gaussian curvature $G_h(r)=-\frac{m''(r)}{m(r)}$ is decreasing along the half meridian. If the Gaussian curvature of $M$ is positive on $r=0$, then the structure of the cut locus $C_q$ of a point $\theta(q)=0$ in $M$ is given as follows:
	\begin{enumerate}
		\item The cut locus $C_q$ is the union of a subarc of the parallel $r=-r(q)$ opposite to $q$ and the meridian opposite to $q$ if $|r(q)<r_0|:=\sup\{r>0|m'(r)<0\}$ and $\varphi(m(r(q)))<\pi$, i.e.
		\begin{equation*}
			C_q=\theta^{-1}(\pi)\cup(r^{-1}(-r(q))\cap\theta^{-1}[\varphi(m(r(q))),2\pi-\varphi(m(r(q)))]).
		\end{equation*}
		\item The cut locus $C_q$ is the meridian $\theta^{-1}(\pi)$ opposite to $q$ if $\varphi(m(r(q)))\geq \pi$ or if $|r(q)|\geq r_0$.
	\end{enumerate}
	Here the function $\varphi(\nu)$ on $(\inf m,m(0))$ is defined as
	\begin{equation*}
		\varphi(\nu):=2\int_{\xi(\nu)}^0\frac{\nu}{m\sqrt{m^2-\nu^2}}dr=2\int_0^{\xi(\nu)}\frac{\nu}{m\sqrt{m^2-\nu^2}}dr,
	\end{equation*}
	where $\xi(\nu):=\min\{r>0|m(r)=\nu\}$.
\end{theorem}

\begin{Remark}
	\begin{enumerate}
		\item
		It is easy to see that if the Gauss curvature $G_h<0$ everywhere, then $h$-geodesics cannot have conjugate points. It follows that in the case the $h$-cut locus of a point $p\in M$ is the opposite meridian to the point.
		
		\item
		See \cite{C2} for a more general class of Riemannian cylinders of revolution whose cut locus can be determined.
	\end{enumerate}
\end{Remark}


\section{Randers rotational metrics}\label{sec_Randers_metric}

\subsection{The navigation with wind $\widetilde{W}=A(r)\frac{\partial}{\partial r}+B\frac{\partial}{\partial \theta}$} \label{sec_A(r),B}

Let $(M,h)$ be the Riemannian metric \eqref{eq_Riemannian_metric_h} on the topological cylinder $M=\{(r,\theta):r\in\R, \ \theta\in[0,2\pi)\}$ such that the Gaussian curvature $G_h\neq 0$, i.e. $m(r)$ is not linear function. We will make this assumption all over the paper. 

\begin{Proposition}\label{thm_A(r)}
	Let $(M,h)$ be the topological cylinder $\R\times \Sph^1$ with its Riemannian metric $h$ and let $\widetilde{W}=A(r)\frac{\partial}{\partial r}+B\frac{\partial}{\partial \theta}$, be a vector filed on $M$ where $A=A(r)$ is smooth function on $\R$, $B$ constant, such that $A^2(r)-B^2m^2(r)<1$. Then
	\begin{enumerate}
		\item[(i)] The solution of the Zermelo's navigation problem for $(M,h)$ and wind $\widetilde{W}$ is the Randers metric $\widetilde{F}=\widetilde{\alpha}+\widetilde{\beta}$, where
		\begin{equation}\label{eq_tilde_a}
			(\widetilde{a}_{ij})=\frac{1}{\Lambda^2}\begin{pmatrix}
				1-B^2m^2(r) & BA(r)m^2(r) \\
				BA(r)m^2(r) ) & m^2(r)(1-A^2(r))
			\end{pmatrix}, \
			(\widetilde{b}_i)=\frac{1}{\Lambda}\begin{pmatrix}
				-A(r) \\ -Bm^2(r)
			\end{pmatrix},
		\end{equation}
		and $\Lambda:=1-\|\widetilde{W}\|_h^2=1-A^2(r)-B^2m^2(r)>0$.
		\item[(ii)] The solution of Zermelo's navigation problem for the data $(M,h)$ and wind $V=A(r)\frac{\partial}{\partial r}$, $A^2(r)<1$ is the Randers metric $F=\alpha+\beta$, where
		\begin{equation}\label{eq_a}
			(a_{ij})=\frac{1}{\lambda^2}\begin{pmatrix}
				1 & 0 \\ 0 & \lambda m^2(r)
			\end{pmatrix}, \
			(b_i)=\frac{1}{\lambda}\begin{pmatrix}
				-A(r) \\ 0
			\end{pmatrix},
		\end{equation}
		and $\lambda:=1-\|V\|_h^2=1-A^2(r)>0$.
		\item[(iii)] The solution of Zermelo's navigation problem for $(M,F=\alpha+\beta)$ and wind $W=B\frac{\partial}{\partial \theta}$, $F(-W)<1$ is the Randers metric $\widetilde{F}=\widetilde{\alpha}+\widetilde{\beta}$ given in \eqref{eq_tilde_a}.
	\end{enumerate}
\end{Proposition}

\begin{proof}[Proof of Proposition \ref{thm_A(r)}]
	\begin{enumerate}
		\item[(i)] The solution of Zermelo's navigation problem with $(M,h)$ and $\widetilde{W}=(\widetilde{W}^1,\widetilde{W}^2)=(A(r),B)$ is obtained from \eqref{eq_1.1*} with $\Lambda=1-\|\widetilde{W}\|_h^2=1-A^2(r)-B^2m^2(r)$.
		
		Taking into account that $\widetilde{W}_i=h_{ij}\widetilde{W}^j$ it follows $(\widetilde{W}_1,\widetilde{W}_2)=(A(r),Bm^2(r))$ and a straightforward computation leads to \eqref{eq_tilde_a}.
		
		\item[(ii)] Similar with (i) using $(M,h)$ and $V=(V^1,V^2)=(A(r),0)$, hence $(V_1,V_2)=(A(r),0)$ and $\lambda=1-\|V\|_h^2=1-A^2(r)$.
		
		\item[(iii)] Follows from Theorem \ref{thm_two_steps_Zermelo}. We observe that $\Lambda=1-A^2(r)-B^2m^2(r)>0$ is actually equivalent to $A^2(r)<1$ and $F(-W)<1$.
		
		Indeed, 
		\begin{equation*}
			\begin{split}
				1-A^2(r)-B^2m^2(r)>0 \Rightarrow 1-A^2(r) > B^2m^2(r) > 0 \Rightarrow A^2(r)<1.
			\end{split}
		\end{equation*}
		and
		\begin{equation*}
			\begin{split}
				B^2m^2(r)<1-A^2(r) \Rightarrow 
				\frac{Bm(r)}{\sqrt{1-A^2(r)}}<1 \Rightarrow F(-W)<1,
			\end{split}
		\end{equation*}
		where we use $F(-W)=\sqrt{a_{22}(-B)^2}=\frac{Bm(r)}{\sqrt{1-A^2(r)}}$.	
	\end{enumerate}
\end{proof}

\begin{Remark}\begin{enumerate}
		\item
		Observe that we actually perform a rigid translation of the Riemannian indicatrix $\Sigma_h$ by $\widetilde{W}$, which is actually equivalent to translating $\Sigma_h$ by $V$ followed by the translation of $\Sigma_F$ by $W$ (see Remark \ref{rem_F_positive}).
		\item Observe that the Randers metric given by \eqref{eq_tilde_a} on the cylinder $\R\times \Sph^1$ is rotational invariant, hence $(M,\widetilde{\alpha}+\widetilde{\beta})$ is a Finslerian surface of revolution. This type of Randers metircs are called {\it Randers rotational metrics}.
		Indeed, let us denote $m_F(r):=F(\frac{\partial}{\partial \theta})$. Observe that in the case $A(r)$ is odd or even function, the function  $m_F(r)$ is even function such that $m_F(0)>0$. 
	\end{enumerate}
\end{Remark}

Theorem \ref{thm: F cut points} implies

\begin{theorem}\label{thm_main_1}
	Let $(M,h)$ be the topological cylinder $\R\times \Sph^1$ with the Riemannian metric $h=dr^2+m^2(r)d\theta^2$ and $\widetilde{W}=A(r)\frac{\partial}{\partial r}+B\frac{\partial}{\partial \theta}$, $A^2(r)+B^2m^2(r)<1$. If we denote by $\widetilde{F}=\widetilde{\alpha}+\widetilde{\beta}$ the solution of Zermelo's navigation problem for $(M,h)$ and $\widetilde{W}$, then the followings are true.
	\begin{enumerate}
		\item[(i)] The $\widetilde{F}$-unit speed geodesics $\cP(t)$ are given by
		$$
		\cP(t)=(r(s(t)),\theta(s(t))+B\cdot s(t)),
		$$
		where $\rho(s)=(r(s),\theta(s))$ are $\alpha$-unit speed geodesic and $t=t(s)$ is the parametric change $t=\int_0^s F(\rho(s),\dot{\rho}(s))ds$.
		\item[(ii)] The point $q=\cP(l)$ is conjugate to $\cP(0)=p$ along $\cP$ if and only if $\widehat{q}=(r(q),\theta(q)-Bl)$ is conjugate to $p$ with respect to $\alpha$ along the $\alpha$-geodesic from $p$ to $\widehat{q}$.
		\item[(iii)] The point $\widehat{q}\in Cut_\alpha(p)$ is an $\alpha$-cut point of $p$ if and only if $q=(r(\widehat{q}),\theta(\widehat{q})+Bl)\in Cut_{\widetilde{F}}(p)$, where $l=d_{\widetilde{F}}(p,q)$.
	\end{enumerate}
\end{theorem}

\begin{proof}[Proof of Theorem \ref{thm_main_1}]
	First of all, observe that $V=A(r)\frac{\partial}{\partial r}$ and $W=B\frac{\partial}{\partial \theta}$ satisfy conditions (i), (ii) in the hypothesis of Theorem \ref{thm: F cut points}.
	
	Indeed, since $(M,h)$ is surface of revolution and $V=(A(r),0)$ it results that $\eta=A(r)dr$ is closed form, hence \eqref{eq_d_eta} is satisfied.
	
	Moreover $W=B\frac{\partial}{\partial \theta}$ is obviously Killing field with respect to $h$, and it is trivial to see that $[V,W]=\left[A(r)\frac{\partial}{\partial r},B\frac{\partial}{\partial \theta}\right]=0$.
	
	The statements (i)-(iii) follows now from Theorem \ref{thm: F cut points} and the fact that the flow of $W=B\frac{\partial}{\partial \theta}$ is just $\varphi_t(r,\theta)=(r,\theta+Bt)$ for any $(r,\theta)\in M$, $t\in\R$.
	
	In this case, $\beta(W)=0$, hence $F(-W)=F(W)=\alpha(W)<1$, hence (iii) is necessary and sufficient condition.
\end{proof}

We have reduced the geometry of the Randers type metric $(M,\widetilde{F})$ to the geometry of the Riemannian manifold $(M,\alpha)$, obtained from $(M,h)$ by \eqref{eq_a}.

\begin{Example}
	Let us observe that there are many cylinders $(M,h)$ and winds $\widetilde{W}$ satisfying conditions in Theorem \ref{thm_main_1}.
	
	For instance, let us consider the topological cylinder $\R\times \Sph^1$ with the Riemannian metric $h=dr^2+m^2(r)d\theta^2$ defined using the warp function $m(r)=e^{-r^2}$.
	
	Consider the smooth function $A:\R\to \left(-\frac{1}{\sqrt{2}},\frac{1}{\sqrt{2}}\right)$, $A(r)=\frac{1}{\sqrt{2}}\frac{r}{\sqrt{r^2+1}}$ and any constant $B\in\left(-\frac{1}{\sqrt{2}},\frac{1}{\sqrt{2}}\right)$.
	
	Then $A^2(r)+B^2m^2(r)<\frac{1}{2}+B^2m^2(r)\leq \frac{1}{2}+B^2<1$.
	
	In this case $\widetilde{W}=\widetilde{\alpha}+\widetilde{\beta}$ is given by
	\begin{equation*}
		\begin{split}
			(\widetilde{a}_{ij})=\frac{1}{\Lambda^2}\begin{pmatrix}
				1-B^2e^{-2r^2} & \frac{B}{\sqrt{2}}\frac{re^{-2r^2}}{\sqrt{r^2+1}} \vspace{0.2cm} \\
				\frac{B}{\sqrt{2}}\frac{re^{-2r^2}}{\sqrt{r^2+1}} & \frac{1}{2}\frac{(r^2+2)e^{-2r^2}}{r^2+1}
			\end{pmatrix},\ (\widetilde{b}_i)=\frac{1}{\Lambda}\begin{pmatrix},
				-\frac{1}{\sqrt{2}}\frac{r}{\sqrt{r^2+1}} \vspace{0.2cm} \\ -Be^{-2r^2}
			\end{pmatrix},
		\end{split}
	\end{equation*}
	where $\Lambda=\frac{1}{2}\frac{r^2+2}{r^2+1}-B^2e^{-2r^2}$.
	
	Observe that $F=\alpha+\beta$ is given by
	\begin{equation*}
		\begin{split}
			(a_{ij})=\frac{1}{\lambda}\begin{pmatrix}
				1 & 0 \vspace{0.2cm} \\ 0 & \lambda e^{-r^2}
			\end{pmatrix},\ (b_{i})=\frac{1}{\lambda}\begin{pmatrix}
				-\frac{1}{\sqrt{2}}\frac{r}{\sqrt{r^2+1}} \vspace{0.2cm}  \\ 0
			\end{pmatrix}, \\
		\end{split}
	\end{equation*}
	where $\lambda=\frac{1}{2}\frac{r^2+2}{r^2+1}$.
\end{Example}

Moreover, we have

\begin{Corollary}\label{cor_*2}
	\begin{enumerate}
		\item[(I)] With notations in Theorem \ref{thm_main_1} let us assume that $(M,h)$ has negative curvature everywhere $G_h(r)<0$, for all $r\in \R$.
		
		\quad If there exist a smooth function $A:\R\to(-1,1)$ and a constant $B$ such that $A^2(r)+B^2m^2(r)<1$ and if $G_\alpha(r)<0$ everywhere, then the $\alpha$-cut locus and the $F=\alpha+\beta$ cut locus of a point $p\in M$ is the opposite meridian to the point $p$.
		
		\quad Moreover, the $\widetilde{F}=\widetilde{\alpha}+\widetilde{\beta}$ cut locus of $p=(r_0,\theta_0)$ is the deformed opposite meridian by the flow of the vector field $W=B\frac{\partial}{\partial\theta}$, i.e. $(r(t),\theta_0+\pi+Bt)$, for all $t\in \R$, where $(r(t),\theta_0+\pi)$ is the opposite meridian to $p=(r_0,\theta_0)$, $r(0)=r_0$.
		\item[(II)] With the notations in Theorem \ref{thm_main_1} let us assume that $(M,h)$ has Gaussian curvature $G_h(r)$ decreasing along any half meridian $[0,\infty)$ and $G_h(0)\geq 0$.
		
		\quad If there exist a smooth function $A:\R\to(-1,1)$ and a constant $B$ such that $A^2(r)+B^2m^2(r)<1$, $G_\alpha(r)$ is decreasing along any half meridian and $G_\alpha \geq 0$, then the $\alpha$-cut locus and the $F$-cut locus of a point $p=(r_0,\theta_0)$ is given in Theorem \ref{thm_Riemannian_cut_locus}.
	\end{enumerate}
	
	Moreover, the $\widetilde{F}$-cut locus of $p$ is obtained by the deformation of the cut locus described in Theorem \ref{thm_Riemannian_cut_locus} by the flow of $W=B\frac{\partial}{\partial \theta}$.
	
\end{Corollary}

\begin{proof}[Proof of Corollary \ref{cor_*2}]
	It is trivial by combining Proposition \ref{thm_A(r)} and Theorem \ref{thm_main_1}.
\end{proof}

\begin{Remark}
	It is not trivial to obtain concrete examples satisfying conditions (I) and (II) in Corollary \ref{cor_*2} in the case $A\neq$ constant. We conjecture that such examples exist leaving the concrete construction for a forthcoming research. The case $A=$ constant is treated below.
\end{Remark}


\subsection{The case $\widetilde{W}=A\frac{\partial}{\partial r}+B\frac{\partial}{\partial \theta}$} \label{sec_A,B}

Consider the case $\widetilde{W}=(\widetilde{W}^1,\widetilde{W}^2)=(A,B)$, where $A$ and $B$ are constants, on the topological cylinder $M=\{(r,\theta):r\in \R, \ \theta\in [0,2\pi)\}$. Here $m:\R\to [0,\infty)$ is an even bounded function such that $m^2<\frac{1-A^2}{B^2}$, $|A|<1$, $B\neq 0$.

Proposition \ref{thm_A(r)} and Theorem \ref{thm_main_1} can be easily rewritten for this case by putting $A(r)=A=$ constant. We will not write them again here. 

Instead, let us give some special properties specific to this case. A straightforward computation gives:

\begin{Proposition} Let $(M,h)$ be the Riemannian metric of the cylinder $\R\times\Sph^1$, and let $\widetilde{W}=A\frac{\partial}{\partial r}+B\frac{\partial}{\partial \theta}$, with $A,B$ real constants such that $m^2<\frac{1-A^2}{B^2}$, $|A|<1$, $B\neq 0$.
	
	Then, the followings are true.
	\begin{enumerate}
		\item[(I)] 
		The Gauss curvatures $G_h$ and $G_\alpha$ 
		of $(M,h)$ and $(M,\alpha)$, respectively, are proportional, i.e.
		$$
		G_\alpha(r)=\frac{1}{\lambda^2}G_h(r),
		$$
		where $\alpha$ is the Riemannian metric obtained in the solution of the Zermelo's navigation problem for $(M,h)$ and $V=A\frac{\partial}{\partial r}$. 
		
		\item[(II)] The geodesic flows $S_h$ and $S_\alpha$ of $(M,h)$ and $(M,\alpha)$, respectively,  satisfy
		$$
		S_h=S_\alpha+\Delta,
		$$
		where $\Delta=-2A^2mm''(y^2)^2\frac{\partial}{\partial y^1}$ is the difference vector field on $TM$ endowed with the canonical coordinates $(r,\theta; y^1,y^2)$. 
	\end{enumerate}
\end{Proposition}

Moreover, we have

\begin{theorem}
	In this case, if $(M,h)$ is a Riemannian metric on the cylinder $M=R\times \Sph^1$ with bounded warp function $m(r)<\frac{\sqrt{1-A^2}}{B}$ where $A,B$ are constants, $|A|<1$, $B\neq 0$, and wind $\widetilde{W}=A\frac{\partial}{\partial r}+B\frac{\partial}{\partial \theta}$ then the followings hold good.
	\begin{enumerate}
		\item[(I)] If $G_h(r)<0$ everywhere, then
		\begin{enumerate}
			\item[(i)] the $\alpha$-cut locus of a point $p$ is the opposite meridian.
			\item[(ii)] the $F$-cut locus of a point $p$ is the opposite meridian, where $F=\alpha+\beta$,
			\begin{equation*}
				\begin{split}
					(a_{ij})=\frac{1}{\lambda^2}\begin{pmatrix}
						1 & 0 \\ 0 & \lambda m^2(r)
					\end{pmatrix}\ (b_i)=\frac{1}{\lambda}\begin{pmatrix}
						-A \\ 0
					\end{pmatrix},
				\end{split}
			\end{equation*}
			and $\lambda:=1-\|V\|_h^2=1-A^2>0$.
			\item[(iii)] The $\widetilde{F}$-cut locus of a point $p$ is the twisted opposite meridian by the flow action $\varphi_t(r,\theta)=(r,\theta+Bt)$.
		\end{enumerate}
		\item[(II)] With the notations in Theorem \ref{thm_main_1} let us assume that $(M,h)$ has Gaussian curvature satisfying $G_h(r)$ is decreasing along any half meridian $[0,\infty)$ and $G_h(0)\geq 0$. Then in this case the cut locus of $\widetilde{F}=\widetilde{\alpha}+\widetilde{\beta}$ is a subarc of the opposite meridian is of the opposite parallel deformed by the flow of $W=B\frac{\partial}{\partial \theta}$.
	\end{enumerate}
	
\end{theorem}

\begin{Example}
	There are many examples satisfying this theorem. For instance the choice $m(r)=e^{-r^2}\leq 1$ gives $G_h(r)=-4r^2-2<0$ which is decreasing on $[0,\infty)$ and $G_h(0)=2>0$. Any choice of constants $A, B$ such that $1<\frac{\sqrt{1-A^2}}{B}$, i.e. $A^2+B^2<1$ is suitable (for instance $(A,B)=(\sin \omega,\cos\omega)$ for a fixed angle $\omega$) for $\widetilde{W}$. Many other examples are possible.
\end{Example}

\begin{Remark}
	A similar study can be done for the case $B=0$.
\end{Remark}

\begin{Remark}
	The extension to the Randers case of the Riemannian cylinders of revolution and study of their cut loci in \cite{C2} can be done in a similar manner.
\end{Remark}

\end{document}